\documentclass[11pt]{article}
\usepackage[a4paper,textwidth=405bp,textheight=646bp,centering]{geometry}
\usepackage{lmodern}     %

\usepackage[T1]{fontenc}    %
\usepackage{microtype}      %

\usepackage{xcolor}

\usepackage{mathtools,amssymb}   %
\mathtoolsset{showonlyrefs,showmanualtags}
\usepackage{bm}             %

\usepackage{my_notation}

\usepackage{amsthm}
\usepackage{thmtools}                   %

\usepackage{my_citations}   %
\usepackage[compact]{titlesec}

\titleformat{\section}[runin]
{\normalfont\normalsize\bfseries}
{\thesection.}
{0.6em}
{}
[.]

\titlespacing*{\section}
{0pt}
{1.2\baselineskip plus 0.25\baselineskip minus 0.15\baselineskip}
{0.8em}

\titleformat{\subsection}[runin]
{\normalfont\normalsize\itshape}
{\thesubsection.}
{0.6em}
{}
[.]

\titlespacing*{\subsection}
{0pt}
{0.9\baselineskip plus 0.2\baselineskip minus 0.1\baselineskip}
{0.8em}

\titleformat{\subsubsection}[runin]
{\normalfont\normalsize\itshape}
{}
{0pt}
{}
[.]

\titlespacing*{\subsubsection}
{0pt}
{0.75\baselineskip plus 0.15\baselineskip minus 0.1\baselineskip}
{0.8em}

\titleformat{\paragraph}[runin]
{\normalfont\normalsize\itshape}
{}
{0pt}
{}
[.]

\titlespacing*{\paragraph}
{0pt}
{0.55\baselineskip plus 0.1\baselineskip minus 0.05\baselineskip}
{0.8em}

\usepackage{csquotes}
\usepackage[british]{babel}     %

\usepackage[hidelinks]{hyperref}       %
\usepackage{bookmark}           %
\usepackage{zref-clever}
\let\cref\zcref
\zcsetup{cap,nameinlink=false}                   %

\usepackage{my_theorems}

\usepackage{enumitem}           %

\begin{document}

\begin{center}
{\Large\bfseries Self-normalised Bennett inequalities\\for Hilbert-valued martingales\par}

\vspace{1.2em}

{\normalsize David Janz\par}
\vspace{0.15em}
{\small University of Oxford\par}
{\small \texttt{david.janz@stats.ox.ac.uk}\par}
\end{center}

\begin{abstract}
We prove time-uniform self-normalised Bennett inequalities for a martingale $(M_n)_{n\geq0}$ in a separable Hilbert space, with $M_0=0$ and increments bounded in norm by one. Writing $V_n$ for its predictable covariance process and $h(u)=(1+u)\log(1+u)-u$ for the Bennett rate function, we show that, for every regularisation parameter $\rho>0$, the process
\[
  \exp\curlyb[\Bigg]{\rho h\roundb[\bigg]{\frac{\norm{(V_n+\rho I)^{-1/2}M_n}}{\sqrt\rho}}-\frac12\log\det\roundb[\big]{I+\rho^{-1}V_n}},\qquad n\geq0,
\]
is a nonnegative supermartingale with initial value one, where $\det$ is the Fredholm determinant. Ville's inequality yields time-uniform Bennett and Bernstein bounds for $\norm{(V_n+\rho I)^{-1/2}M_n}$. The result permits conditional covariance increments of infinite rank; in finite dimensions, the resulting boundaries sharpen existing martingale-transform and determinant-based variational bounds. The same construction extends to compensated marked point processes with bounded jumps.

Mixing these supermartingales over $\rho$ gives simultaneous control over the regularisation parameter. Consequences include an upper law of the iterated logarithm for the regularised ellipsoidal radius in separable Hilbert spaces and, in finite dimensions, spectrum-sensitive finite-time bounds and an upper law of the iterated logarithm for the unregularised radius $\norm{V_n^{-1/2}M_n}$, whose constant $1$ is sharp over the class. We also obtain a time-uniform Bernstein inequality for the martingale norm $\norm{M_n}$ with dependence on $\tr(V_n)$ and $\norm{V_n}_{\op}$.
\end{abstract}

\section{Introduction}
Let $(M_n,\cF_n)_{n\ge0}$ be a real-valued martingale with $M_0=0$, increments
$X_n=M_n-M_{n-1}$ satisfying $X_n\le1$ almost surely, and predictable variance
$V_n=\sum_{i=1}^n\E[X_i^2\mid\cF_{i-1}]$. Write $h$ for the Bennett rate function
\[
  h(u)=(1+u)\log(1+u)-u,\qquad u\ge0.
\]
Freedman's martingale extension of Bennett's inequality
\citep{bennett1962probability,freedman1975tail} states that for $v>0$ and $y \geq 0$,
\[
 \P{\exists n\ge1: M_n\ge0,\ v h(M_n/v)\ge y \text{ and } V_n\le v}
  \le e^{-y}.
\]
This bound fixes the variance level $v$ in advance. We instead study self-normalised bounds, in which the scale is the predictable variance process itself \citep{delapena1999general,delapena2004selfnormalized,delapena2009selfnormalized}. For $\rho>0$, we consider the regularised process
\[
  \frac{\abs{M_n}}{\sqrt{V_n+\rho}},\qquad n\geq0.
\]
Our inequality takes the following form in the scalar case: if $\abs{X_n}\leq1$ almost surely, then for every $\rho>0$ and $y\geq0$,
\[
  \P[\bigg]{\exists n\geq1:\rho h\roundb[\bigg]{\frac{1}{\sqrt{\rho}}\,\frac{\abs{M_n}}{\sqrt{V_n+\rho}}}>y+\frac12\log(1+V_n/\rho)}\leq e^{-y}.
\]
Setting $\rho=v$ calibrates the regularisation to the variance level $v$. On $\curlyb{V_n\leq v}$, the Bennett term in the preceding display lies between $\frac12v h(\abs{M_n}/v)$ and $v h(\abs{M_n}/v)$, while the logarithmic correction is at most $\frac12\log2$.

Now let $(M_n)$ be a martingale taking values in a separable real Hilbert space $\cH$, with norm $\norm{\cdot}$, and suppose that its increments satisfy $\norm{X_n}\leq1$ almost surely. Let $V_n$ denote its predictable covariance process, whose values are positive trace-class operators on $\cH$. For a positive trace-class operator $A$, write $\lambda_1(A)\geq\lambda_2(A)\geq\cdots\geq0$ for its eigenvalues, counted with multiplicity and arranged in nonincreasing order; for a positive-definite matrix $A$, write $\kappa(A)=\lambda_{\max}(A)/\lambda_{\min}(A)$. The scalar ratio $\abs{M_n}/\sqrt{V_n+\rho}$ is replaced by the ellipsoidal radius
\[
  \sup_{\norm{u}=1}\frac{\abs{\langle u,M_n\rangle}}{\sqrt{\langle u,V_nu\rangle+\rho}}=\norm{(V_n+\rho I)^{-1/2}M_n},
\]
and the scalar logarithmic correction $\frac12\log(1+V_n/\rho)$ becomes
\[
  \frac12\log\det(I+\rho^{-1}V_n).\]
These replacements give a Hilbert-space Bennett bound with the same form as the scalar inequality above. The bound follows by applying Ville's inequality to the nonnegative supermartingale constructed in our main theorem, \cref{thm:bennett-logdet-potential}; a standard relaxation of the Bennett rate gives a Bernstein bound.

The same ellipsoidal radius and log-determinant correction are classical under conditional sub-Gaussian assumptions, with the quadratic rate $u^2/2$ in place of $h(u)$ \citep{delapena2004selfnormalized,delapena2009selfnormalized,abbasiyadkori2013online}. Bennett and Bernstein analogues have mainly been developed in two settings. Martingale-transform inequalities assume increments of the form $X_n=\epsilon_nx_n$, where $x_n$ is predictable \citep{faury2020improved}; extensions to general separable Hilbert spaces retain this structure \citep{akhavan2025bernstein,metelli2025generalized,martineztaboada2026vector}. Consequently, each conditional covariance increment has rank at most one. A second line combines exponential-moment bounds for scalar projections through covering or variational arguments \citep{whitehouse2026timeuniform,ziemann2025vector,chugg2025variational}; in the Bennett and Bernstein settings, however, all currently available inequalities without martingale-transform structure are finite-dimensional.

By contrast, our main theorem applies to general bounded increments in a separable Hilbert space, including martingales whose conditional covariance increments have infinite rank; the countable-alphabet and random-curve examples in \cref{sec:applications} illustrate this scope. In finite dimensions, our Bernstein boundary is strictly smaller than that of \citet{faury2020improved}, while our Bennett and Bernstein boundaries lie strictly below the determinant-based PAC--Bayes line boundaries of \citet{chugg2025variational}.

The supermartingale construction also extends to compensated marked point processes (\cref{prop:continuous-bounded-jumps}). Mixing the discrete-time supermartingales indexed by $\rho$ gives a bound that holds simultaneously over all regularisation scales (\cref{thm:mixed-rho-bennett}), allowing the scale to depend on the predictable covariance process $V_n$. A principal consequence is a spectrum-sensitive finite-dimensional upper law of the iterated logarithm (LIL) (\cref{cor:finite-dimensional-lil}): almost surely on the event that $\lambda_{\min}(V_n)\to\infty$ and $\log\det\roundb[\big]{\lambda_{\min}(V_n)^{-1}V_n}=o\roundb[\big]{\lambda_{\min}(V_n)}$,
\[
  \limsup_{n\to\infty}\frac{\norm{V_n^{-1/2}M_n}}{\sqrt{2\log\log\lambda_{\min}(V_n)+\log\det\roundb[\big]{\lambda_{\min}(V_n)^{-1}V_n}}}\leq1.
\]
The finite-dimensional LIL of \citet{whitehouse2026timeuniform} uses $2\log\log\lambda_{\max}(V_n)+d\log\kappa(V_n)$ under the square root. The normalising term in the preceding display is no larger, and is strictly smaller unless $V_n$ is a scalar multiple of the identity. In one dimension, the determinant term vanishes and the display recovers the classical self-normalised upper LIL. Beyond this finite-dimensional LIL, the mixed bound also yields a regularised upper LIL in arbitrary separable Hilbert spaces (\cref{cor:regularised-lil}) and finite-time spectrum-sensitive bounds for the unregularised self-normalised radius (\cref{cor:condition-number-bound}). It also gives a time-uniform Bernstein inequality for $\norm{M_n}$ in terms of $\tr(V_n)$ and $\norm{V_n}_{\op}$ (\cref{prop:covariance-sensitive-bernstein}).

\section{Preliminaries}
\label{sec:prelims}

We collect the operator notation used throughout and record basic properties of the log-determinant correction.

\subsection{Trace-class operators and conditional covariance}
Throughout, $\cH$ is a separable real Hilbert space. For $x,y\in\cH$, write
$x\otimes y$ for the rank-one operator $(x\otimes y)u=\langle x,u\rangle y$.
Then $x\otimes x$ is positive and $\tr(x\otimes x)=\norm{x}^2$. We write
$\trclass$ for the trace-class operators on $\cH$ and $\ptrclass$
for the positive trace-class operators. If $A,B$ are bounded self-adjoint operators,
$A\preceq B$ means $\langle u,Au\rangle\le\langle u,Bu\rangle$ for all $u\in\cH$. For $A\in\trclass$, write $\det(I+A)$ for the Fredholm determinant. When $A$ is self-adjoint,
\[
  \det(I+A)=\prod\nolimits_j(1+\lambda_j(A)),
\]
where the nonzero eigenvalues are counted with multiplicity. If $G\in\ptrclass$, then $\det(I+G)$ is strictly positive, while if $S\in\ptrclass$ and $\norm{S}_{\op}<1$, then $\det(I-S)$ is strictly positive. For the standard properties of the Fredholm determinant, see \citet[Ch.~3]{simon2005trace}.

If $X$ is an $\cH$-valued random variable with $\E\norm{X}^2<\infty$, then $X\otimes X$ is Bochner integrable in $\trclass$, since $\norm{X\otimes X}_{\trclass}=\norm{X}^2$. For a sub-$\sigma$-algebra $\cG$, write $\E[X\otimes X\mid\cG]$ for its conditional Bochner expectation. It belongs to $\ptrclass$ a.s., is characterized by
\[
  \langle u,\E[X\otimes X\mid\cG]v\rangle=\E[\langle u,X\rangle\langle X,v\rangle\mid\cG], \qquad u,v\in\cH,
\]
and satisfies $\tr\E[X\otimes X\mid\cG]=\E[\norm{X}^2\mid\cG]$ a.s.

\subsection{The log-determinant correction} For $\rho>0$ and $V\in\ptrclass$, define
\[
  \ell_\rho(V)=\frac12\log\det(I+\rho^{-1}V).
\]
When $\cH=\Rd$, $\ell_\rho(V)$ is the log-volume ratio of the ellipsoids associated with the positive definite matrices $\rho I$ and $V+\rho I$:
\[
  \ell_\rho(V) = \log \frac{\vol\curlyb{x\in\Rd:x\tran \rho I x\le1}}{\vol\curlyb{x\in\Rd:x\tran (V+\rho I)x\le1}}.
\]
Writing $\lambda_1,\dots,\lambda_d$ for the eigenvalues of $V$, Jensen's inequality also gives
\[
  \ell_\rho(V) = \frac{1}{2}\sum_{j=1}^d\log\roundb[\bigg]{1+\frac{\lambda_j}{\rho}} \le \frac{d}{2}\log\roundb[\bigg]{1+\frac{\tr V}{d\rho}}.
\]
The log-determinant correction is classical in multivariate self-normalisation; see, for example, \citet[Cor.~4.3]{delapena2004selfnormalized} and \citet[Thm.~14.7]{delapena2009selfnormalized}.

\section{Fixed-regularisation Bennett inequalities}\label{sec:main-results}

This section develops the fixed-$\rho$ theory. We first construct a nonnegative supermartingale and derive the corresponding time-uniform Bennett and Bernstein inequalities. We then compare these bounds with martingale-transform and directional sub-$\psi$ inequalities, extend the construction to compensated marked point processes, and give three applications illustrating its scope.

For $\rho>0$, $m\in\cH$ and $G\in\ptrclass$, define the potential function
\[
  U_\rho^h(m,G)=\exp\curlyb[\bigg]{\rho h\roundb[\bigg]{\frac{\norm{(G+\rho I)^{-1/2}m}}{\sqrt\rho}}-\ell_\rho(G)}.
\]
The following theorem gives the supermartingale underlying the fixed- and mixed-regularisation bounds.

\begin{theorem}\label{thm:bennett-logdet-potential}
Let $(M_n,\cF_n)_{n\ge0}$ be a martingale in $\cH$ with $M_0=0$ and increments
\[
  X_n=M_n-M_{n-1},
  \qquad
  \E[X_n\mid\cF_{n-1}]=0,
  \qquad
  \norm{X_n}\le1\quad\text{a.s.}
\]
Write
\[
  V_n=\sum_{i=1}^n\E[X_i\otimes X_i\mid\cF_{i-1}],
  \qquad n\ge0.
\]
Then, for every $\rho>0$,
\[
  \roundb{U_\rho^h(M_n,V_n)}_{n\ge0}
\]
is a nonnegative supermartingale with initial value one.
\end{theorem}

The proof is given in \cref{sec:proof-bennett-logdet-potential}. For infinitesimal increments, a second-order expansion yields a one-step supermartingale bound. After centring and taking expectations, the quadratic terms agree with the first-order change in the regularised radius and log-determinant correction as the covariance is updated. A scalar inequality for $h$, together with convexity along the covariance interpolation, extends this local calculation to the full bounded-increment range.

The family of supermartingales in \cref{thm:bennett-logdet-potential} is mixed in \cref{sec:mixing-results} to obtain bounds that hold simultaneously over the regularisation scale.

\begin{remark}[Gaussian analogue]
The classical quadratic log-determinant potential is obtained by Gaussian mixing under a conditional sub-Gaussian increment condition. Let $(X_n,\cF_n)_{n\ge1}$ be a martingale-difference sequence in $\cH$, and let $W_n\in\ptrclass$ be $\cF_{n-1}$-measurable and satisfy
\[
  \E\left[\exp\curlyb{\langle\theta,X_n\rangle}\,\middle|\,\cF_{n-1}\right]
  \le
  \exp\curlyb[\bigg]{\frac12\langle\theta,W_n\theta\rangle}
  \quad\text{a.s. for every $\theta\in\cH$.}
\]
Write $g(u) = u^2/2$, $M_n=\sum_{i=1}^nX_i$, $G_n=\sum_{i=1}^nW_i$ and let $\gamma_\rho$ denote the centred Gaussian cylindrical measure on $\cH$ with covariance $\rho^{-1}I$. Interpreting the integral through finite-dimensional projections, define
\begin{align*}
  U_\rho^g(M_n,G_n)
  &:=
  \int_{\cH}\exp\curlyb[\bigg]{\langle\theta,M_n\rangle-\frac12\langle\theta,G_n\theta\rangle}\,\gamma_\rho(\dif\theta) \\
  &=
  \exp\curlyb[\bigg]{\frac12\norm{(G_n+\rho I)^{-1/2}M_n}^{2}-\ell_\rho(G_n)}.
\end{align*}
Then $\roundb{U_\rho^g(M_n,G_n)}_{n\ge0}$ is a nonnegative supermartingale. In finite dimensions, the same Gaussian-mixture calculation appears in \citet[Eqns.~4.18--4.20]{delapena2004selfnormalized}; see also the multivariate canonical formulation of \citet[Thm.~14.7]{delapena2009selfnormalized}.
\end{remark}

\begin{remark}[Sharpness of the log-determinant coefficient]\label{rem:scalar-calibration}
The coefficient $1/2$ multiplying $\log\det(I+\rho^{-1}G)$ in the Bennett potential $U_\rho^h$ cannot be reduced. Fix $\rho>0$, replace this coefficient by $\alpha>0$, and consider the one-step martingale with $X=\epsilon$ and $X=-\epsilon$ with probability $1/2$ each. Then $\abs{M_1}=\epsilon$, $V_1=\epsilon^2$, and the logarithm of the modified one-step potential is
\[
  \rho h\roundb[\bigg]{\frac{\epsilon}{\sqrt{\rho(\rho+\epsilon^2)}}}
  -\alpha\log(1+\epsilon^2/\rho)
  =
  \roundb[\bigg]{\frac12-\alpha}\frac{\epsilon^2}{\rho}
  +o(\epsilon^2),
  \qquad \epsilon\downarrow0.
\]
If $\alpha<1/2$, the modified potential is greater than one after one step, almost surely, for all sufficiently small $\epsilon$. Since its initial value is one, it fails to be a supermartingale.
\end{remark}

For each fixed $\rho>0$, Ville's inequality gives the following time-uniform Bennett bound.

\begin{corollary}\label{cor:bennett}
Under the assumptions of \cref{thm:bennett-logdet-potential}, for every $\rho>0$ and $y\ge0$,
\[
  \P[\bigg]{\exists n\ge1\colon \rho h\roundb[\bigg]{\frac{\norm{(V_n+\rho I)^{-1/2}M_n}}{\sqrt\rho}}>y+\ell_\rho(V_n)}\le e^{-y}.
\]
\end{corollary}

\begin{proof}
By \cref{thm:bennett-logdet-potential}, $(U_\rho^h(M_n,V_n))_{n\ge0}$ is a nonnegative supermartingale with initial value one, and Ville's inequality gives
\[
  \P{\exists n\ge0\colon U_\rho^h(M_n,V_n)>e^y}\le e^{-y}.
\]
The result follows by expanding the definition of $U_\rho^h$.
\end{proof}

Define
\[
  \psi_h(s)=e^s-1-s,\quad s \in[0,\infty);\qquad \psi_b(s)=\frac{s^2}{2(1-s/3)},\quad s \in [0,3).
\]
For a function $\psi$, write $\psi^\ast$ for its convex conjugate. Thus $h=\psi_h^\ast$; write $b=\psi_b^\ast$ for the Bernstein rate function. The following is the Bernstein relaxation of \cref{cor:bennett}.

\begin{corollary}\label{cor:bernstein}
Under the assumptions of \cref{thm:bennett-logdet-potential}, for every $\rho>0$ and $y\geq0$,
\[
  \P[\bigg]{\exists n\geq1\colon \norm{(V_n+\rho I)^{-1/2}M_n}>\sqrt{2\roundb[\big]{y+\ell_\rho(V_n)}}+\frac{y+\ell_\rho(V_n)}{3\sqrt\rho}}\leq e^{-y}.
\]
\end{corollary}

\begin{proof}
Since $\psi_h\leq\psi_b$ on $[0,3)$, convex conjugacy gives $b(u)\leq h(u)$ for $u\geq0$. Hence \cref{cor:bennett} remains valid with $h$ replaced by $b$. The function $b$ is strictly increasing, and direct calculation gives
\[
  b^{-1}(x)=\sqrt{2x}+\frac{x}{3},\qquad x\geq0.
\]
Applying this identity with $x=(y+\ell_\rho(V_n))/\rho$ proves the result.
\end{proof}

Since boundedness implies directional exponential-moment conditions, one might ask whether these conditions suffice to make the Bennett potential above, or the analogous potential obtained from the Bernstein rate, a supermartingale. The following remark shows that they do not.

\begin{remark}[Directional relaxations of boundedness]\label{rem:directional-subpsi}
Under the martingale-difference assumption, the bound $\norm{X_n}\leq1$ implies, for every $u\in\cH$ with $\norm{u}=1$ and every $s\in\Rp$,
\[
  \log\E\left[\exp\curlyb{s\langle u,X_n\rangle}\,\middle|\,\cF_{n-1}\right]\leq\psi_h(s)\E\left[\langle u,X_n\rangle^2\,\middle|\,\cF_{n-1}\right]\quad\text{a.s.}
\]
Indeed, this follows by taking conditional expectations in the scalar inequality
\[
  e^{sx}\leq1+sx+x^2\psi_h(s),\qquad \abs{x}\leq1,\quad s\in\Rp,
\]
and using $\log(1+x)\leq x$. Since $\psi_h(s)\leq\psi_b(s)$ for $0\leq s<3$, the corresponding directional conditional sub-$\psi_b$ inequalities hold with the same predictable variance.

For the Bernstein rate, define
\[
  U_\rho^b(m,G)=\exp\curlyb[\bigg]{\rho b\roundb[\bigg]{\frac{\norm{(G+\rho I)^{-1/2}m}}{\sqrt\rho}}-\ell_\rho(G)}.
\]
One may ask whether the directional conditional sub-$\psi_h$ inequalities suffice to make $\roundb{U_\rho^h(M_n,V_n)}_{n\geq0}$ a nonnegative supermartingale, or whether the corresponding sub-$\psi_b$ inequalities suffice for $\roundb{U_\rho^b(M_n,V_n)}_{n\geq0}$. Neither is true. \cref{prop:sub-poisson-counterexample} gives a centred scalar random variable $X$, with $v=\E X^2$, satisfying
\[
  \log\E e^{sX}=\log\E e^{-sX}\leq v\psi_h(s),\qquad s\in\Rp,
\]
and hence the corresponding sub-$\psi_b$ inequalities for $0\leq s<3$, but such that
\[
  \E U_\rho^h(X,v)>1 \spaced{and} \E U_\rho^b(X,v)>1
\]
for all sufficiently large $\rho$.
\end{remark}

\subsection{Comparison with related self-normalised inequalities}

We first consider inequalities for martingale transforms, whose increments are conditionally supported on predictable subspaces of dimension at most one. We then turn to directional sub-$\psi$ inequalities, where determinant-based variational arguments yield line boundaries that can be compared directly with \cref{cor:bennett,cor:bernstein}.

\paragraph{Martingale-transform setting}

Martingale-transform inequalities concern processes of the form
\[
  M_n=\sum_{i=1}^n\epsilon_i x_i,\qquad V_n=\sum_{i=1}^n\E[\epsilon_i^2\mid\cF_{i-1}]\,x_i\otimes x_i,
\]
where $x_i\in\cH$ is $\cF_{i-1}$-measurable and $\epsilon_i$ is a real-valued $\cF_i$-measurable martingale difference. This structure arises naturally in sequential generalised linear models; see \cref{ex:logistic}.

The seminal result in the martingale transform setting is the finite-dimensional Bernstein inequality of \citet[Thm.~1]{faury2020improved}, which further assumes that $\cH=\Rd$ and $\norm{\epsilon_i x_i}\leq1$ almost surely. Writing $L_n=y+\ell_\rho(V_n)$, their result gives
\[
  \P[\bigg]{\exists n\geq1\colon \norm{(V_n+\rho I)^{-1/2}M_n}>\frac{\sqrt\rho}{2}+\frac{2}{\sqrt\rho}\roundb[\big]{L_n+d\log2}}\leq e^{-y}.
\]
By comparison, \cref{cor:bernstein} gives the strictly smaller threshold $\sqrt{2L_n}+L_n/(3\sqrt\rho)$.

Recent work extends this martingale-transform framework to separable Hilbert spaces. \citet{akhavan2025bernstein,metelli2025generalized} treat bounded scalar noise, while \citet{martineztaboada2026vector} allow broader conditions on the scalar noise~$\epsilon_i$. Their formulations use different normalisers and variance processes, so a direct quantitative comparison with \cref{cor:bennett,cor:bernstein} is not immediate. All three nevertheless retain increments of the form $X_i=\epsilon_i x_i$, so the conditional law of $X_i$ is supported on a predictable subspace of dimension at most one. Consequently, each conditional covariance increment has rank at most one, and these results do not cover martingales with infinite-rank conditional covariance increments.

This restriction excludes the countable-alphabet and random-curve martingales discussed in \cref{sec:applications}; see \cref{ex:alphabet,ex:curves}. Their conditional covariance increments may have infinite rank, and both satisfy the assumptions of \cref{thm:bennett-logdet-potential}, whereas none of the martingale-transform inequalities cited above applies.

\paragraph{Directional sub-$\psi$ setting}
A second line of work derives self-normalised inequalities from exponential bounds in every fixed direction.

Let $\psi:[0,s_{\max})\to\Rp$. We call an adapted pair $(M_n,V_n)_{n\geq0}$, with $M_n$ taking values in $\cH$ and $V_n$ a positive self-adjoint operator, directionally sub-$\psi$ if, for every $u\in\cH$ with $\norm{u}=1$ and $0\leq s<s_{\max}$, the process
\[
  \exp\curlyb[\big]{s\langle u,M_n\rangle-\psi(s)\langle u,V_nu\rangle},\qquad n\geq0,
\]
is dominated by a nonnegative supermartingale with initial value one. Bounded increments imply the directional sub-$\psi_h$ and sub-$\psi_b$ conditions, also called the directional sub-Poisson and sub-gamma conditions, respectively.

Multivariate self-normalisation under related canonical assumptions goes back to \citet{delapena2004selfnormalized,delapena2009selfnormalized}. In the sub-Gaussian case, the log-determinant inequality extends to separable Hilbert spaces; see \citet[Cor.~3.5]{abbasiyadkori2013online}. General directional sub-Poisson and sub-gamma inequalities have so far only been obtained in finite dimensions. \citet{whitehouse2026timeuniform} obtain condition-number-dependent bounds through a covering argument; we compare these with our mixed-regularisation results in \cref{sec:condition-num}. A determinant-based variational approach was developed by \citet{ziemann2025vector} in the martingale-transform setting and extended by \citet{chugg2025variational} to general directional sub-$\psi$ processes.

Among results beyond the martingale-transform setting, \citet{chugg2025variational} provide a particularly strong determinant-based baseline: their variational inequality applies to general directional sub-$\psi$ processes and retains the log-determinant geometry.

\begin{remark}[Comparison with PAC--Bayes line boundaries]\label{rem:chugg-pac-bayes-comparison}
We specialise the variational inequality in \citet[Eq.~(B.8)]{chugg2025variational} to the functions $\psi_h$ and $\psi_b$.

Fix $\cH=\Rd$, $\rho>0$, $y\geq0$ and $f\in\curlyb{h,b}$. Write $\psi_f$ for the corresponding function and $s_f$ for the right endpoint of its domain, so that $f=\psi_f^\ast$, $s_h=\infty$ and $s_b=3$. Fix $0<a\leq\rho$ and $0<s<s_f$. For $n\geq1$, write $L_n=y+\ell_\rho(V_n)$ and define $q_d=d/(d+2)$,
\[
  \chi_{n,\rho}=1-\sqrt{\frac{\rho}{\rho+\lambda_{\min}(V_n)}}, \qquad B_{n,\rho}^{f}(a,s)=\frac{L_n+a\roundb[\big]{q_d+\chi_{n,\rho}^2}\psi_f(s)}{s\sqrt a\,\chi_{n,\rho}},
\]
with the convention $B_{n,\rho}^{f}(a,s)=\infty$ when $\chi_{n,\rho}=0$. Their argument gives\footnotemark{}
\[
  \P[\bigg]{\exists n\geq1\colon \norm{(V_n+\rho I)^{-1/2}M_n}>B_{n,\rho}^{f}(a,s)}\leq e^{-y}.
\]
For each fixed $a,s$, every linear-in-$L_n$ boundary $B_{n,\rho}^f(a,s)$ lies strictly above the corresponding curved boundary in \cref{cor:bennett,cor:bernstein}.

Indeed, if $L_n=0$, then $y=0$ and $V_n=0$, so $\chi_{n,\rho}=0$ and $B_{n,\rho}^{f}(a,s)=\infty$. Fix $n$ with $L_n>0$. The claim is immediate when $B_{n,\rho}^{f}(a,s)=\infty$. Otherwise, set $x=L_n/\rho$ and $r=\chi_{n,\rho}\sqrt{a/\rho}\in(0,1]$. Since $q_d>0$ and $\psi_f(s)>0$, while the monotonicity of $s\mapsto\psi_f(s)/s^2$ gives $r^2\psi_f(s)\geq\psi_f(sr)$, we obtain
\begin{align*}
    \frac{B_{n,\rho}^{f}(a,s)}{\sqrt\rho}
  =\frac{x+r^2\roundb[\big]{1+q_d/\chi_{n,\rho}^2}\psi_f(s)}{sr}
  &>\frac{x+r^2\psi_f(s)}{sr}
  \geq\frac{x+\psi_f(sr)}{sr}\\
  &\qquad\qquad\geq\inf_{0<z<s_f}\frac{x+\psi_f(z)}{z}
  =f^{-1}(x).
\end{align*}
The final identity is the variational representation of $f^{-1}$. Multiplying by $\sqrt\rho$ gives
\[
  B_{n,\rho}^{f}(a,s)>\sqrt\rho\,f^{-1}\roundb[\bigg]{\frac{L_n}{\rho}}.
\]
The right-hand side is exactly the Bennett threshold in \cref{cor:bennett} when $f=h$, and equals $\sqrt{2L_n}+L_n/(3\sqrt\rho)$, the Bernstein threshold in \cref{cor:bernstein}, when $f=b$.
\end{remark}

\footnotetext{Starting from \citet[Eq.~(B.8)]{chugg2025variational}, take $U_0=\rho I$ and $\delta=e^{-y}$, and choose the prior and posterior to be uniform on ellipsoids with shape matrices $a(\rho I)^{-1}$ and $a(V_n+\rho I)^{-1}$, respectively. Retaining the posterior-covariance factor $d/(d+2)$ and using the mean shift from their containment argument gives the displayed threshold. In arXiv v2, the final specialisation writes the ellipsoid shape using the functional inverse $\psi_f^{-1}(s)$, while the subsequent covariance calculation uses the reciprocal $1/\psi_f(s)$; retaining $a$ avoids this mismatch.}

\citet{chugg2025variational} combine these line boundaries by stitching over geometric log-determinant epochs. More generally, fixing the dual parameter in variational or exponential-tilt arguments yields linear boundaries \citep{howard2020timeuniform}, while geometric stitching adapts their slopes across variance scales \citep{howard2021timeuniform}. By contrast, \cref{thm:bennett-logdet-potential} gives the curved Bennett boundary directly through a single supermartingale, and its Bernstein relaxation requires no stitching.

\subsection{Continuous-time self-normalised Bennett inequality}
\label{sec:continuous-bounded-jumps}

The fixed-$\rho$ potential extends to compensated marked point processes with bounded jumps. In the scalar case, \cref{cor:bennett} is the self-normalised analogue of Freedman's Bennett-form martingale inequality recalled in the introduction. \citet[Cor.~3.4(iii)]{dzhaparidze2001bernstein} establish the corresponding continuous-time inequality with predictable quadratic variation restricted to a level fixed in advance; the proposition below gives its Hilbert-valued self-normalised extension.

\begin{proposition}[restate=continuousBoundedJumps,name=]\label{prop:continuous-bounded-jumps}
Let $(\cF_t)_{t\geq0}$ be a filtration satisfying the usual conditions, let $\cZ$ be a measurable space, and let $\mu$ be an integer-valued random measure on $\Rp\times\cZ$ with predictable compensator $\nu$. Assume that, on an event of probability one, $\nu(\{t\}\times\cZ)=0$ and $\mu(\{t\}\times\cZ)\leq1$ for every $t\geq0$. Let $H=(H_t(z))_{t\geq0,\,z\in\cZ}$ be predictable and $\cH$-valued, with $\norm{H_t(z)}\leq1$, and suppose that, for every $T<\infty$,
\[
  \int_{(0,T]\times\cZ}\norm{H_t(z)}^2\,\nu(\dif t,\dif z)<\infty \qquad\text{a.s.}
\]
Define
\[
  M_t=\int_{(0,t]\times\cZ}H_s(z)\,(\mu-\nu)(\dif s,\dif z), \qquad V_t=\int_{(0,t]\times\cZ}H_s(z)\otimes H_s(z)\,\nu(\dif s,\dif z).
\]
Then $M$ is a purely discontinuous local martingale with $\norm{\Delta M_t}\leq1$ for every $t\geq0$, and $V$ is its operator-valued predictable quadratic variation. For every $\rho>0$, $(U_\rho^h(M_t,V_t))_{t\geq0}$ is a nonnegative supermartingale with initial value one. Consequently, for every $y\geq0$,
\[
  \P[\bigg]{\exists t\geq0:\rho h\roundb[\bigg]{\frac{\norm{(V_t+\rho I)^{-1/2}M_t}}{\sqrt\rho}}>y+\ell_\rho(V_t)}\leq e^{-y}.
\]
\end{proposition}

The proof is given in \cref{sec:proof-continuous-bounded-jumps}. An infinitesimal limit of the one-step estimate underlying \cref{thm:bennett-logdet-potential} yields a compensator drift inequality. In finite dimensions, combining this inequality with an Itô--Meyer formula gives the supermartingale property; finite-rank projection then extends the result to a general separable Hilbert space. A~Bernstein relaxation follows as in discrete time.

For continuous local martingales, \citet[Cor.~4.3]{delapena2004selfnormalized} establish the following exact quadratic log-determinant identity. Let $M$ be an $\Rd$-valued continuous local martingale with $M_0=0$ and predictable quadratic variation $V$. Suppose that $\lambda_{\min}(V_t)\to\infty$ almost surely and $\E\exp{\langle\theta,V_t\theta\rangle}<\infty$ for every $\theta\in\Rd$ and $t>0$. Then, for every $\rho>0$ and $y>0$,
\[
  \P[\bigg]{\exists t\geq0\colon \frac12\norm{(V_t+\rho I)^{-1/2}M_t}^2\geq y+\frac12\log\det(I+\rho^{-1}V_t)}=e^{-y}.
\]
This has the same regularised radius and log-determinant correction as \cref{prop:continuous-bounded-jumps}, with the quadratic rate replacing the Bennett rate; it applies to continuous local martingales under the stated asymptotic and integrability conditions.

The mixing construction developed in the next section applies equally to the continuous-time supermartingales above. In particular, it yields bounds on $\norm{M_t}$ in terms of $\norm{V_t}_{\op}$ and $\tr(V_t)$, in the spirit of \citet[Thm.~5.1]{kallenberg1991dimensionfree}. We state only the corresponding discrete-time result in \cref{sec:ordinary-norm-bounds}.

\subsection{Applications and scope}\label{sec:applications}

We give three examples illustrating both the motivation for and the scope of the fixed-$\rho$ inequalities. Sequential logistic regression is a canonical martingale-transform application. The countable-alphabet and random-curve examples have conditional covariance increments that may have infinite rank: they are covered by \cref{thm:bennett-logdet-potential}, whereas none of the martingale-transform or directional sub-$\psi$ inequalities discussed above applies. Each example also has a direct marked-point-process formulation covered by \cref{prop:continuous-bounded-jumps}; we do not state these formulations separately.

\begin{example}[Sequential logistic regression]\label{ex:logistic}
Let $\sigma(u)=(1+e^{-u})^{-1}$ and fix $\theta_\star\in\cH$. Let $x_i\in\cH$ be $\cF_{i-1}$-measurable with $\norm{x_i}\leq1$, and let $Y_i$ be an $\cF_i$-measurable Bernoulli random variable satisfying $\P{Y_i=1\mid\cF_{i-1}}=\sigma(\langle\theta_\star,x_i\rangle)=:m_i$. Define
\[
  M_n=\sum_{i=1}^n(Y_i-m_i)x_i,\qquad V_n=\sum_{i=1}^nm_i(1-m_i)x_i\otimes x_i.
\]
Then $(M_n)_{n\geq0}$ is a martingale with predictable covariance process $V_n$, and its increments are bounded by one. Moreover, $M_n$ and $V_n$ are respectively the gradient and negative Hessian of the cumulative logistic log-likelihood at $\theta_\star$. Thus \cref{cor:bennett} gives time-uniform control of the likelihood gradient in the exact curvature of the likelihood.

Under a known bound $\norm{\theta_\star}\leq b$, Theorem~3.1 of \citet[]{akhavan2025bernstein} converts the resulting concentration inequality into a time-uniform sequence of confidence ellipsoids for $\theta_\star$, centred at the ridge-regularised logistic maximum-likelihood estimator.
\end{example}

We now turn to two applications beyond the martingale-transform setting.

\begin{example}[Countable alphabets]\label{ex:alphabet}
Let $Z_i$ be an $\Np$-valued $\cF_i$-measurable random variable, and write
\[
  p_i(j)=\P{Z_i=j\mid\cF_{i-1}},\qquad p_i=(p_i(j))_{j\geq1}.
\]
Let $(e_j)_{j\geq1}$ be the standard orthonormal basis of $\ell^2(\Np)$, and set $X_i=\frac{e_{Z_i}-p_i}{\sqrt2}$. Then $\E[X_i\mid\cF_{i-1}]=0$ and $\norm{X_i}\leq1$. Writing $\widehat p_n=n^{-1}\sum_{i=1}^ne_{Z_i}$, the corresponding martingale and predictable covariance process are
\[
  M_n=\frac{n}{\sqrt2}\roundb[\bigg]{\widehat p_n-\frac1n\sum_{i=1}^np_i},\qquad V_n=\frac12\sum_{i=1}^n\roundb[\big]{\diag(p_i)-p_i\otimes p_i}.
\]
\cref{cor:bennett} gives time-uniform control of the empirical mass function around the average conditional mass function in its exact predictable covariance geometry.

If $p_i$ has infinite support, then $\diag(p_i)$ has infinite rank while $p_i\otimes p_i$ has rank one, so the conditional covariance operator of $X_i$ has infinite rank.

Countably supported models arise, for example, in infinite-occupancy and infinite-species problems \citep{karlin1967central,gnedin2007notes,anevski2017estimating}.
\end{example}

\begin{example}[Random curves]\label{ex:curves}
Let $\cT$ be a compact interval, and let $f_i$ be $\cF_i$-measurable random elements of $L^2(\cT)$ satisfying, for some $B>0$,
\[
  \E[f_i\mid\cF_{i-1}]=0,\qquad \norm{f_i}_{L^2(\cT)}\leq B\quad\text{almost surely}.
\]
Set
\[
  M_n=\frac1B\sum_{i=1}^nf_i,\qquad V_n=\frac1{B^2}\sum_{i=1}^n\E[f_i\otimes f_i\mid\cF_{i-1}].
\]
\cref{cor:bennett} gives time-uniform self-normalised control of the partial sums.

The conditional covariance operator may have infinite rank after a single observation. For example, let $\cT=[0,1]$, let $g\in L^2([0,1])$ be real-valued and mean-zero, and extend $g$ $1$-periodically to $\R$. Suppose that $g$ has infinitely many nonzero Fourier coefficients, and let $U_i$ be $\cF_i$-measurable and conditionally uniform on $[0,1]$ given $\cF_{i-1}$. Taking
\[
  f_i=g(\,\cdot-U_i)
\]
gives $\E[f_i\mid\cF_{i-1}]=0$ and $\norm{f_i}_{L^2([0,1])}=\norm{g}_{L^2([0,1])}$ almost surely. The conditional covariance operator is convolution with the autocorrelation of $g$. It is diagonal in the Fourier basis, and its eigenvalue at a given frequency is positive whenever the corresponding Fourier coefficient of $g$ is nonzero. Hence it has infinite rank.

Random curves in $L^2(\cT)$ are studied throughout functional data analysis; see \citet{hall2006properties,hoermann2010weakly,panaretos2013fourier}.
\end{example}

\section{Mixing over the regularisation scale}\label{sec:mixing-results} The fixed-$\rho$ inequalities require the regularisation scale to be chosen in advance. We now mix the family of supermartingales from \cref{thm:bennett-logdet-potential} to obtain a bound that holds simultaneously over candidate scales. This permits the scale to be selected from the predictable covariance process. We use this to derive spectrum-sensitive laws of the iterated logarithm and finite-time bounds for the unregularised self-normalised radius. A final consequence gives a Bernstein inequality for the martingale norm with separate dependence on $\tr(V_n)$ and~$\norm{V_n}_{\op}$.

Let $\pi$ be a Borel probability measure on $\Rpp$. For $a>0$ and $u>1$, define the interval-mass penalty
\[
  \Delta_\pi(a,u)=-\log\pi([a,ua]),
\]
with the convention $-\log0=\infty$. For $\rho>0$, write
\[
  R_{n,\rho}=\norm{(V_n+\rho I)^{-1/2}M_n}.
\]
The following theorem gives the general mixed-regularisation bound.

\begin{theorem}\label{thm:mixed-rho-bennett}
Under the assumptions of \cref{thm:bennett-logdet-potential}, for every $y\ge0$,
\[
  \P[\bigg]{\exists n\ge1,\ \exists a>0,\ \exists u>1: ua h\roundb[\bigg]{\frac{R_{n,a}}{u\sqrt{a}}} > y+\ell_a(V_n)+\Delta_\pi(a,u)} \le e^{-y}.
\]
\end{theorem}
\begin{proof}
For each $\rho>0$, \cref{thm:bennett-logdet-potential} gives that
$(U_\rho^h(M_n,V_n))_{n\ge0}$ is a nonnegative supermartingale with initial value one. Hence
\[
  \cL_n=\int_0^\infty U_\rho^h(M_n,V_n)\,\pi(\dif\rho)
\]
is a nonnegative supermartingale with $\cL_0=1$. By Ville's inequality, it remains to show that the
event in the theorem is contained in $\{\exists n\ge1:\cL_n>e^y\}$.

Suppose the event occurs, and choose $n,a,u$ for which the displayed inequality holds. Then
$\Delta_\pi(a,u)<\infty$. By the conjugacy of $h$ with $s\mapsto e^s-1-s$ on $\Rp$, there is an
$s>0$ such that
\[
  s\sqrt a R_{n,a}-ua(e^s-1-s) > y+\ell_a(V_n)+\Delta_\pi(a,u).
\]
For every $\rho\in[a,ua]$,
$\sqrt\rho R_{n,\rho}\ge\sqrt a R_{n,a}$, $\rho(e^s-1-s)\le ua(e^s-1-s)$, and
$\ell_\rho(V_n)\le\ell_a(V_n)$. Taking logarithms in the definition of $U_\rho^h$ gives
\[
  \log U_\rho^h(M_n,V_n) \ge s\sqrt a R_{n,a}-ua(e^s-1-s)-\ell_a(V_n) > y+\Delta_\pi(a,u).
\]
Therefore
\[
  \cL_n \ge \int_{[a,ua]}U_\rho^h(M_n,V_n)\,\pi(\dif\rho) > \pi([a,ua])e^{y+\Delta_\pi(a,u)} = e^y.
\]
Ville's inequality gives the claim.
\end{proof}

The consequences below use the following logarithmic mixing law $\varpi$ on $\Rpp$, or a rescaling thereof:
\[
  \varpi(\dif\rho)=\frac{(e+1)\,\dif\rho}{2\rho(e+\abs{\log\rho})\{e+\log(e+\abs{\log\rho})\}^2}, \qquad \rho>0.
\]
The next lemma records its interval mass and the resulting asymptotic penalty.
\begin{lemma}\label{lem:logarithmic-mixing}
Fix $u>1$. For every $a\geq1$, writing $x=\log a$ and $c=\log u$,
\[
  \varpi([a,ua])=\frac{e+1}{2}\roundb[\bigg]{\frac{1}{e+\log(e+x)}-\frac{1}{e+\log(e+x+c)}}.
\]
The map $a\mapsto\Delta_\varpi(a,u)$ is nondecreasing on $[1,\infty)$. Moreover,
\[
  \Delta_\varpi(a,u)=\log\log a+2\log\log\log a+O_u(1), \qquad a\to\infty.
\]
\end{lemma}

\begin{proof}
By the change of variables $x=\log\rho$ and symmetry,
\[
  \int_0^\infty\frac{\dif\rho}{\rho(e+\abs{\log\rho})\{e+\log(e+\abs{\log\rho})\}^2}=2\int_0^\infty\frac{\dif x}{(e+x)\{e+\log(e+x)\}^2}=\frac{2}{e+1},
\]
which verifies the normalisation. For $a\geq1$, the same change of variables gives
\begin{align*}
  \varpi([a,ua])&=\frac{e+1}{2}\int_x^{x+c}\frac{\dif t}{(e+t)\{e+\log(e+t)\}^2}\\
  &=\frac{e+1}{2}\roundb[\bigg]{\frac{1}{e+\log(e+x)}-\frac{1}{e+\log(e+x+c)}}.
\end{align*}
The integrand is nonincreasing on $[0,\infty)$, so its integral over $[x,x+c]$ is nonincreasing in $x$. Since $x=\log a$ is increasing in $a$, the map $a\mapsto\Delta_\varpi(a,u)$ is nondecreasing.

For fixed $c>0$, uniformly over $t\in[x,x+c]$,
\[
  \frac{1}{(e+t)\{e+\log(e+t)\}^2}=\frac{1+o(1)}{x(\log x)^2}, \qquad x\to\infty.
\]
Consequently,
\[
  \varpi([a,ua])=\frac{(e+1)c}{2x(\log x)^2}\roundb{1+o(1)}.
\]
Taking negative logarithms and substituting $x=\log a$ and $c=\log u$ proves the stated asymptotic.
\end{proof}

\subsection{Laws of the iterated logarithm}

The logarithmic mixing law yields two laws of the iterated logarithm. The first is a regularised upper LIL for martingales in any separable Hilbert space, with a normalisation that adapts to the spectrum of $V_n$. In~finite dimensions, choosing the regularisation asymptotically negligible relative to $\lambda_{\min}(V_n)$ gives a spectrum-sensitive LIL for the unregularised self-normalised radius $\norm{V_n^{-1/2}M_n}$. In one dimension, this recovers the classical self-normalised upper LIL.

\begin{corollary}[restate=corRegularisedLIL,name=]\label{cor:regularised-lil}
Under the assumptions of \cref{thm:bennett-logdet-potential}, almost surely, every sequence $(a_n)_{n\geq1}$ with $a_n>e$ satisfying
\[
  a_n\to\infty, \qquad \log\det\roundb[\big]{I+a_n^{-1}V_n}=o(a_n)
\]
also satisfies
\[
  \limsup_{n\to\infty}\frac{R_{n,a_n}}{\sqrt{2\log\log a_n+\log\det\roundb[\big]{I+a_n^{-1}V_n}}}\leq1.
\]
\end{corollary}

Since \cref{thm:mixed-rho-bennett} holds simultaneously over $a>0$, the sequence $(a_n)$ may be random; in particular, it may be chosen as a measurable function of $(M_n,V_n)$.

\begin{proof}
For $k\in\Np$, let $\cE_k$ be the event that, simultaneously for all $n\geq1$, $a>0$ and $u>1$,
\[
  ua h\roundb[\bigg]{\frac{R_{n,a}}{u\sqrt a}}\leq k+\Delta_\varpi(a,u)+\ell_a(V_n).
\]
By \cref{thm:mixed-rho-bennett}, $\P{\cE_k}\geq1-e^{-k}$, and hence $\cE=\bigcup_{k\geq1}\cE_k$ has full probability. On $\cE$, there is a finite random integer $K$ for which the preceding inequality holds with $k=K$.

Let $(a_n)$ satisfy the assumptions of the corollary, fix $u>1$, and write
\[
  L_n=2\log\log a_n+2\ell_{a_n}(V_n)=2\log\log a_n+\log\det\roundb[\big]{I+a_n^{-1}V_n}.
\]
Then $L_n\to\infty$ and $L_n=o(a_n)$. By \cref{lem:logarithmic-mixing},
\[
  K+\Delta_\varpi(a_n,u)+\ell_{a_n}(V_n)=\frac12L_n+o(L_n),
\]
since $L_n\geq2\log\log a_n$ and $\log\log\log a_n=o(\log\log a_n)$. Substituting $a=a_n$ in the mixed inequality and inverting $h$ gives
\[
  R_{n,a_n}\leq u\sqrt{a_n}\,h^{-1}\roundb[\bigg]{\frac{L_n+o(L_n)}{2ua_n}}=\roundb[\big]{\sqrt u+o(1)}\sqrt{L_n},
\]
where the equality uses $L_n=o(a_n)$ and $h^{-1}(x)=\sqrt{2x}(1+o(1))$ as $x\downarrow0$. Thus the limsup is at most $\sqrt u$. Letting $u\downarrow1$ proves the claim.
\end{proof}

\begin{remark}[LIL under polynomial spectral decay]
  Suppose that, almost surely for all sufficiently large $n$, the eigenvalues of $V_n$ satisfy
\[
  \lambda_j(V_n)\leq A n j^{-\alpha}, \qquad j\geq1,
\]
for some $A>0$ and $\alpha>1$. Using $\int_0^\infty\log(1+x^{-\alpha})\,\dif x=\pi/\sin(\pi/\alpha)$, an integral comparison gives, almost surely for all sufficiently large $n$ and every $a>0$,
\[
  \ell_a(V_n)\leq\frac{\pi}{2\sin(\pi/\alpha)}\roundb{A n/a}^{1/\alpha}.
\]
Consequently, the regularised LIL (\cref{cor:regularised-lil}) applies whenever $a_n\to\infty$ and $a_n/n^{1/(\alpha+1)}\to\infty$. Taking $a_n=n^\beta$ with $1/(\alpha+1)<\beta<1$ gives
\[
  \limsup_{n\to\infty}\frac{R_{n,a_n}}{n^{\frac{1-\beta}{2\alpha}}}\leq\sqrt{\frac{\pi A^{1/\alpha}}{\sin(\pi/\alpha)}} \qquad\text{a.s.}
\]
The restriction $\beta<1$ ensures that $\log\log a_n=o\roundb{n^{(1-\beta)/\alpha}}$.
\end{remark}

In finite dimensions, taking $a_n=o(\lambda_{\min}(V_n))$ makes the regularised and unregularised self-normalised radii asymptotically equivalent.

\begin{corollary}[restate=corFiniteDimensionalLIL,name=]\label{cor:finite-dimensional-lil}
Assume $\cH=\Rd$. Then, almost surely on the event
\[
  \curlyb[\Big]{\lambda_{\min}(V_n)\to\infty,\ \log\det\roundb[\big]{\lambda_{\min}(V_n)^{-1}V_n}=o\roundb[\big]{\lambda_{\min}(V_n)}},
\]
one has
\[
  \limsup_{n\to\infty}\frac{\norm{V_n^{-1/2}M_n}}{\sqrt{2\log\log\lambda_{\min}(V_n)+\log\det\roundb[\big]{\lambda_{\min}(V_n)^{-1}V_n}}}\leq1.
\]
\end{corollary}

The passage from \cref{cor:regularised-lil} to this result uses the smallest eigenvalue only to choose $a_n=o(\lambda_{\min}(V_n))$ and compare the regularised radius with $\norm{V_n^{-1/2}M_n}$. If $V\in\ptrclass$ acts on an infinite-dimensional Hilbert space, then $V\succeq \epsilon I$ cannot hold for any $\epsilon>0$, so the same conversion from the regularised to the unregularised radius is unavailable.

\begin{proof}
Work on a sample path on which \cref{cor:regularised-lil} holds and the event in the statement occurs. Write
\[
  \gamma_n=\lambda_{\min}(V_n), \qquad L_n=2\log\log\gamma_n+\log\det(\gamma_n^{-1}V_n).
\]
Then $L_n\to\infty$ and $L_n=o(\gamma_n)$. Define
\[
  r_n=\min\curlyb[\Big]{\sqrt{\frac{\gamma_n}{L_n}},\,\exp\curlyb{\sqrt{L_n}}}, \qquad a_n=\frac{\gamma_n}{r_n}.
\]
Both terms in the minimum defining $r_n$ tend to infinity, so $r_n\to\infty$ and $a_n/\gamma_n\to0$. Moreover,
\[
  \frac{L_n}{a_n}=\frac{r_nL_n}{\gamma_n}\leq\sqrt{\frac{L_n}{\gamma_n}}\to0, \qquad \log r_n\leq\sqrt{L_n}=o(L_n).
\]
The choice of $r_n$ also gives $a_n\geq\sqrt{\gamma_nL_n}$, and hence $\sqrt{\gamma_n}\leq a_n\leq\gamma_n$ eventually.

Let $\lambda_1(V_n),\dots,\lambda_d(V_n)$ be the eigenvalues of $V_n$. Since $\lambda_j(V_n)/\gamma_n\geq1$ and $r_n\to\infty$,
\[
  2\ell_{a_n}(V_n)=\sum_{j=1}^d\log\roundb[\bigg]{1+r_n\frac{\lambda_j(V_n)}{\gamma_n}}=\log\det(\gamma_n^{-1}V_n)+d\log r_n+o(1).
\]
Also, $\sqrt{\gamma_n}\leq a_n\leq\gamma_n$ gives $\log\log a_n=\log\log\gamma_n+O(1)$. It follows that
\[
  2\log\log a_n+2\ell_{a_n}(V_n)=L_n+o(L_n), \qquad \ell_{a_n}(V_n)=o(a_n).
\]
Applying \cref{cor:regularised-lil} gives
\[
  \limsup_{n\to\infty}\frac{R_{n,a_n}}{\sqrt{L_n}}\leq1.
\]

Finally, $V_n\succeq\gamma_nI$ implies $V_n+a_nI\preceq\roundb[\big]{1+\frac{a_n}{\gamma_n}}V_n$, and hence
\[
  \norm{V_n^{-1/2}M_n}\leq\sqrt{1+\frac{a_n}{\gamma_n}}\,R_{n,a_n}.
\]
Since $a_n/\gamma_n\to0$, the result follows.
\end{proof}

\begin{remark}[Scalar case and sharpness]\label{rem:scalar-lil-calibration}
When $d=1$, the determinant term vanishes and \cref{cor:finite-dimensional-lil} gives
\[
  \limsup_{n\to\infty}\frac{\abs{M_n}}{\sqrt{2V_n\log\log V_n}}\leq1 \qquad\text{a.s. on }\curlyb{V_n\to\infty}.
\]
This upper bound also follows from \citet[Lem.~1.6 and Cor.~4.2]{delapena2004selfnormalized}; their lemma permits a conditional Bernstein moment condition. For the simple symmetric random walk, the limsup in the display is equal to $1$ almost surely.
\end{remark}

\begin{remark}[Condition-number LIL comparison]\label{rem:whitehouse-lil-comparison}
Recall that $\kappa(V_n)=\lambda_{\max}(V_n)/\lambda_{\min}(V_n)$. The finite-dimensional LIL of \citet[Cor.~4.5]{whitehouse2026timeuniform} uses the normalising term
\[
  \sqrt{2\log\log\lambda_{\max}(V_n)+d\log\kappa(V_n)}.
\]
The normalising term in \cref{cor:finite-dimensional-lil} is no larger, since
\[
  2\log\log\lambda_{\min}(V_n)+\log\det\roundb[\big]{\lambda_{\min}(V_n)^{-1}V_n}\leq2\log\log\lambda_{\max}(V_n)+ (d-1)\log\kappa(V_n).
\]
The inequality is strict whenever $V_n$ is not a scalar multiple of the identity.
\end{remark}

\subsection{Finite-time spectrum-sensitive bounds}\label{sec:condition-num}

Regularising at a multiple of $\lambda_{\min}(V_n)$ also gives a finite-time bound for the unregularised self-normalised radius. The bound retains the full normalised spectrum of $V_n$; relaxing the log-determinant term yields a condition-number inequality in the style of \citet{whitehouse2026timeuniform}.

\begin{corollary}[name=]\label{cor:condition-number-bound}
Let $\pi$ be a Borel probability measure on $\Rpp$ and assume $\cH=\Rd$. For every $y\geq0$, with probability at least $1-e^{-y}$, simultaneously for all $n\geq1$ with $\lambda_{\min}(V_n)>0$, $r>0$ and $u>1$, writing
\[
  L_{n,r,u}=y+\Delta_\pi(r\lambda_{\min}(V_n),u)+\frac12\log\det\roundb[\big]{I+(r\lambda_{\min}(V_n))^{-1}V_n},
\]
one has
\[
  \norm{V_n^{-1/2}M_n}\leq\sqrt{1+r}\roundb[\bigg]{\sqrt{2uL_{n,r,u}}+\frac{L_{n,r,u}}{3\sqrt{r\lambda_{\min}(V_n)}}}.
\]
\end{corollary}

\begin{proof}
Since $\lambda_{\min}(V_n) I\preceq V_n$, one has $V_n+r\lambda_{\min}(V_n) I\preceq(1+r)V_n$, and hence
\[
  \norm{V_n^{-1/2}M_n}\leq\sqrt{1+r}\,\norm{(V_n+r\lambda_{\min}(V_n) I)^{-1/2}M_n}.
\]
Fix $y\geq0$. By \cref{thm:mixed-rho-bennett}, there is an event $\cE$ of probability at least $1-e^{-y}$ on which, simultaneously for all $n\geq1$, $a>0$ and $u>1$,
\[
  ua h\roundb[\bigg]{\frac{R_{n,a}}{u\sqrt a}}\leq y+\ell_a(V_n)+\Delta_\pi(a,u).
\]
Work on $\cE$, fix $n\geq1$ such that $\lambda_{\min}(V_n)>0$, and fix $r>0$ and $u>1$. Taking $a=r\lambda_{\min}(V_n)$ and using $h^{-1}(x)\leq\sqrt{2x}+x/3$ gives
\[
  \norm{(V_n+r\lambda_{\min}(V_n) I)^{-1/2}M_n}=R_{n,r\lambda_{\min}(V_n)}\leq\sqrt{2uL_{n,r,u}}+\frac{L_{n,r,u}}{3\sqrt{r\lambda_{\min}(V_n)}}.
\]
Since $n$, $r$ and $u$ were arbitrary, the inequality holds simultaneously on $\cE$.
\end{proof}

The following remark first instantiates the condition-number inequality of \citet{whitehouse2026timeuniform} in the bounded-increment setting, and then relaxes the spectrum-sensitive terms in \cref{cor:condition-number-bound} to recover a condition-number bound of similar form.

\begin{remark}[Comparison with condition-number bounds]\label{rem:whitehouse-condition-number}
Fix $\rho_0>0$ and $y\geq0$, and, for a positive-definite matrix $A$, write $\kappa(A)=\lambda_{\max}(A)/\lambda_{\min}(A)$. Define
\[
  L_n^\kappa=y+2\log\log\roundb[\bigg]{\frac{e\lambda_{\max}(V_n\vee\rho_0I)}{\rho_0}}+\frac d2\log\kappa(V_n\vee\rho_0I)+C_d,
\]
where $\vee$ is taken spectrally and $C_d>0$ depends only on $d$. With simple parameter choices,\footnotemark{} \citet[Thm.~4.1]{whitehouse2026timeuniform} give, in the bounded-increment setting considered here, with probability at least $1-e^{-y}$, simultaneously for all $n\geq1$,
\[
  \norm{(V_n\vee\rho_0I)^{-1/2}M_n}\leq2\sqrt{2eL_n^\kappa}+\frac{2eL_n^\kappa}{3\sqrt{\lambda_{\min}(V_n\vee\rho_0I)}}.
\]
When $\lambda_{\min}(V_n)\geq\rho_0$, one has $V_n\vee\rho_0I=V_n$, so this controls $\norm{V_n^{-1/2}M_n}$. \cref{cor:condition-number-bound} controls the same radius at every positive-definite time. To compare the corresponding $L_{n,1,u}$ with $L_n^\kappa$, take $\pi$ to be the image of $\varpi$ under the map $\rho\mapsto\rho_0\rho$, fix $u>1$, and consider in turn the log-determinant and interval-mass terms.

\footnotetext{Set their $\rho=\rho_0$, $\delta=e^{-y}$, $\alpha=e$, $\beta=2$, and $\epsilon=1/2$, take their stitching function to be $k\mapsto\zeta(2)(k+1)^2$, and specialise to the sub-gamma scale $1/3$. Bound the minimum cardinality of the Euclidean $\eta$-cover of the unit sphere in $\Rd$ required by $c_d(3/\eta)^{d-1}$, where $c_d>0$ only depends on $d$.}

The log-determinant contribution to $L_{n,1,u}$ is
\[
  \frac12\log\det\roundb[\big]{I+\lambda_{\min}(V_n)^{-1}V_n}=\frac12\sum_{j=1}^d\log\roundb[\bigg]{1+\frac{\lambda_j(V_n)}{\lambda_{\min}(V_n)}}.
\]
This term retains the full normalised spectrum of $V_n$, and may therefore be substantially smaller than the condition-number relaxation
\[
  \frac12\log\det\roundb[\big]{I+\lambda_{\min}(V_n)^{-1}V_n}\leq\frac{d-1}{2}\log\kappa(V_n)+\frac d2\log2.
\]
After this relaxation, the resulting $\log\kappa(V_n)$ dependence is of the same form as that in $L_n^\kappa$.

The remaining covariance-dependent contribution to $L_{n,1,u}$ is the interval-mass penalty. By the exact formula in \cref{lem:logarithmic-mixing}, there is a constant $C_u>0$ such that, whenever $\lambda_{\min}(V_n)\geq\rho_0$,
\[
  \Delta_\pi(\lambda_{\min}(V_n),u)=\Delta_\varpi(\lambda_{\min}(V_n)/\rho_0,u)\leq2\log\log\roundb[\bigg]{\frac{e\lambda_{\min}(V_n)}{\rho_0}}+C_u.
\]
Thus, whereas the iterated-logarithm term in $L_n^\kappa$ is evaluated at $\lambda_{\max}(V_n)$, the corresponding term in $L_{n,1,u}$ is evaluated at the smaller spectral scale $\lambda_{\min}(V_n)$.

Relaxing the log-determinant term therefore recovers a condition-number bound of similar form. The unrelaxed bound in \cref{cor:condition-number-bound} retains the full normalised eigenvalue profile of $V_n$.
\end{remark}

\subsection{Bernstein inequality for the martingale norm}\label{sec:ordinary-norm-bounds}

Finally, we use the mixed bound to derive a time-uniform Bernstein inequality for $\norm{M_n}$ in terms of $\tr(V_n)$ and $\norm{V_n}_{\op}$. The result is in the style of dimension-free norm inequalities for martingales in smooth Banach spaces \citep{pinelis1994optimum}; its trace and operator-norm dependence is analogous to that of the finite-dimensional sub-Gaussian confidence spheres of \citet{chugg2025confidencespheres}.

For $y\geq0$, $n\geq1$ and $u>1$, define
\[
  L_{n,u}=1\vee\curlyb[\big]{y+\Delta_\varpi(1\vee\tr(V_n),u)}.
\]

\begin{proposition}\label{prop:covariance-sensitive-bernstein}
Under the assumptions of \cref{thm:mixed-rho-bennett}, for every $y\geq0$, with probability at least $1-e^{-y}$, simultaneously for all $n\geq1$,
\[
  \norm{M_n}\leq\inf_{u>1}\curlyb[\bigg]{\sqrt{4u\roundb[\big]{\tr(V_n)\vee\curlyb[\big]{\roundb{1\vee\norm{V_n}_{\op}}L_{n,u}}}+2u\tr(V_n)}+\frac{L_{n,u}}{\sqrt2}}.
\]
\end{proposition}

The scalar-total-variance inequalities of \citet{pinelis1994optimum} and their time-uniform extensions in \citet{howard2020timeuniform} yield stitched boundaries of order
\[
  \sqrt{\tr(V_n)\curlyb[\big]{y+\log\log\tr(V_n)}}+y+\log\log\tr(V_n).
\]
By contrast, the preceding proposition has order
\[
  \sqrt{\tr(V_n)}+\sqrt{\roundb[\big]{1\vee\norm{V_n}_{\op}}\curlyb[\big]{y+\log\log\tr(V_n)}}+y+\log\log\tr(V_n),
\]
up to lower-order iterated logarithms and universal constants. It can therefore be smaller by an iterated-logarithmic factor when $V_n$ has effective rank substantially larger than the confidence penalty, while giving the same order in rank-one regimes.

\begin{proof}
By \cref{thm:mixed-rho-bennett}, with probability at least $1-e^{-y}$, simultaneously for all $n\geq1$, $a>0$ and $u>1$,
\[
  ua h\roundb[\bigg]{\frac{R_{n,a}}{u\sqrt a}}\leq y+\Delta_\varpi(a,u)+\ell_a(V_n).
\]
Work on this event and fix $n\geq1$ and $u>1$. For any $a>0$, $V_n \preceq\norm{V_n}_{\op}I$ gives \[\norm{M_n}\leq\sqrt{\norm{V_n}_{\op}+a}\,R_{n,a}.\] Inverting $h$ in the mixed inequality and using $h^{-1}(x)\leq\sqrt{2x}+x/3$, we obtain
\begin{align*}
  \norm{M_n}
  &\leq\sqrt{2u\roundb[\big]{\norm{V_n}_{\op}+a}\curlyb[\big]{y+\Delta_\varpi(a,u)+\ell_a(V_n)}}\\
  &\qquad+\frac{\sqrt{\norm{V_n}_{\op}+a}}{3\sqrt a}\curlyb[\big]{y+\Delta_\varpi(a,u)+\ell_a(V_n)}.
\end{align*}

We now choose
\[
  a=\roundb[\big]{1\vee\norm{V_n}_{\op}}\vee\frac{\tr(V_n)}{L_{n,u}}.
\]
The two terms in this choice ensure that $a\geq1\vee\norm{V_n}_{\op}$ and $\tr(V_n)/a\leq L_{n,u}$, respectively. Since $L_{n,u}\geq1$ and $\norm{V_n}_{\op}\leq\tr(V_n)$, one also has $1\leq a\leq1\vee\tr(V_n)$. The monotonicity of the interval-mass penalty therefore gives $y+\Delta_\varpi(a,u)\leq L_{n,u}$. Using $\ell_a(V_n)\leq\tr(V_n)/(2a)$, together with $\norm{V_n}_{\op}+a\leq2a$, the preceding norm bound becomes
\begin{align*}
  \norm{M_n}
  &\leq\sqrt{4ua\roundb[\bigg]{L_{n,u}+\frac{\tr(V_n)}{2a}}}+\frac{\sqrt2}{3}\roundb[\bigg]{L_{n,u}+\frac{\tr(V_n)}{2a}}\\
  &\leq\sqrt{4uaL_{n,u}+2u\tr(V_n)}+\frac{L_{n,u}}{\sqrt2}.
\end{align*}
Finally, $aL_{n,u}=\tr(V_n)\vee\curlyb[\big]{\roundb{1\vee\norm{V_n}_{\op}}L_{n,u}}$. Substituting this identity and taking the infimum over $u>1$ proves the result.
\end{proof}

\section{Proofs of the supermartingale results} We first prove the discrete-time supermartingale theorem through a one-step inequality. The continuous-time proof then obtains the required compensator drift inequality as an infinitesimal consequence of the same estimate.

\subsection{Discrete-time Bennett potential (\cref{thm:bennett-logdet-potential})}
\label{sec:proof-bennett-logdet-potential}

Fix $\rho>0$. Since $U_\rho^h\ge0$ and $U_\rho^h(M_0,V_0)=U_\rho^h(0,0)=1$, it suffices to prove that, for every $n\ge1$,
\[
  \E[U_\rho^h(M_n,V_n)\mid\cF_{n-1}]
  \le
  U_\rho^h(M_{n-1},V_{n-1}).
\]
For $n\ge1$, write $W_n=\E[X_n\otimes X_n\mid\cF_{n-1}]$, so that $M_n=M_{n-1}+X_n$ and $V_n=V_{n-1}+W_n$. Conditionally on $\cF_{n-1}$, the required inequality follows from the following one-step estimate: if $X$ is a centred $\cH$-valued random variable with $\norm{X}\le1$ a.s. and $W=\E[X\otimes X]$, then, for every $m\in\cH$ and $G\in\ptrclass$,
\begin{equation}\label{eq:one-step}
  \E\,U_\rho^h(m+X,G+W)\le U_\rho^h(m,G).\tag{$\dagger$}
\end{equation}
We prove this estimate. Define
\[
  F(x)=\exp\{\rho h(\norm{x})\},\qquad x\in\cH.
\]
We first establish the following pointwise inequality.

\begin{lemma}[restate=bennettTangent,name=]\label{lem:bennett-tangent}
For every $x,y\in\cH$ with $\norm{y}\le\rho^{-1}$,
\[
  F(x+y)
  \le
  \roundb[\Big]{1+\frac{\rho}{2}\norm{y}^2} F(x)
  +
  \roundb[\Big]{1+\frac{\rho}{2}\langle x,y\rangle} \langle \nabla F(x),y\rangle.
\]
\end{lemma}

\begin{proof}
  Since $h(u) = O(u^2)$ as $u \downarrow 0$, one has $F(x) - F(0) = O(\norm{x}^2)$, so $F$ is Fréchet differentiable with $\nabla F(0) = 0$.

  Suppose that $x\ne0$. Then, by the chain rule,
  \[
    \nabla F(x) = \rho\log(1+\norm{x})F(x)\frac{x}{\norm{x}}.
  \]
  Dividing the desired inequality by $F(x)$ shows that it is equivalent to
\[
  \exp\curlyb{\rho (h(\norm{x+y}) - h(\norm{x}))} \leq 1 + \frac{\rho}{2}\norm{y}^2 + \rho \frac{\log(1+\norm{x})}{\norm{x}} \langle x,y \rangle + \frac{\rho^2}{2} \frac{\log(1+\norm{x})}{\norm{x}} \langle x, y \rangle^2.
\]
The left-hand side depends on $x$ and $y$ only through the initial radius $\norm{x}$ and the terminal radius $\norm{x+y}$. On the right-hand side, the increment $y$ appears through $\langle x,y \rangle$ and $\norm{y}^2$. These are related by
\[
  \norm{x+y}^2 = \norm{x}^2 + 2\langle x,y \rangle + \norm{y}^2.
\]
Consequently, once $r:=\norm{x}$ and $q:=\norm{x+y}$ are fixed, either of the quantities $\langle x,y \rangle$ and $\norm{y}^2$ determines the other. Eliminating $\norm{y}^2$, the right-hand side becomes the convex quadratic in $a= \langle x/\norm{x}, y\rangle$ given by
\[
  D(a) = 1+\frac{\rho}{2}(q^2 - r^2) + \rho (\log(1+r) - r) a + \frac{\rho^2 r\log(1+r)}{2} a^2.
\]
Let $b \geq 0$ be defined by $\norm{y}^2 = a^2 + b^2$. Then $q^2 = (r+a)^2 + b^2$ and $\norm{y}^2 \leq \rho^{-2}$ gives
\[
  -q-r \leq a \leq q-r \spaced{and} a \geq \frac{q^2-r^2-\rho^{-2}}{2r}.
\]
Let $\cI=[a_-,a_+]$ be the interval given by the above two constraints and write
\[
   a_\star = \frac{r-\log(1+r)}{\rho r\log(1+r)}
\]
for the unconstrained minimiser of $D$. Since $a\in\cI$, the desired inequality holds if
\[
  \exp\curlyb{\rho (h(q) - h(r))} \leq \min_{u\in\cI} D(u).
\]
And since $D$ is convex and $\cI$ is an interval, it suffices to check that the inequality holds at the left endpoint $a=a_-$, at the right endpoint $a = a_+$ and at the unconstrained minimiser $a = a_\star$ when $a_\star\in\cI$. \cref{app:scalar-bellman} verifies these three scalar inequalities.

For the case $x = 0$, apply the inequality already proved to any sequence $x_n \neq 0$ with $x_n \to 0$ and pass to the limit. By the continuity of $F$ and $\nabla F$,
\[
  F(y) \leq 1 + \frac{\rho}{2} \norm{y}^2,
\]
which is the desired inequality at $x=0$.
\end{proof}

We now apply the above pointwise inequality after normalising by the updated covariance. Put
\[
  B=G+W+\rho I,\qquad z=\frac{B^{-1/2}m}{\sqrt\rho},\qquad Y=\frac{B^{-1/2}X}{\sqrt\rho},\qquad
  S=B^{-1/2}WB^{-1/2}.
\]
Then $\E Y=0$, $S=\rho\E[Y\otimes Y]$, and $\norm{Y}\le\rho^{-1}$ a.s. Moreover,
\[
  I-S=B^{-1/2}(G+\rho I)B^{-1/2},
\]
so $S\in\ptrclass$ and $\norm{S}_{\op}<1$.

Applying \cref{lem:bennett-tangent} with $x=z$ and $y=Y$, and taking expectations, yields
\begin{align*}
    \E F(z+Y) &\leq F(z) +\frac{\rho}{2}F(z)\E\norm{Y}^2 +\frac{\rho}{2}\E[\langle\nabla F(z),Y\rangle\langle z,Y\rangle] \\
    &= F(z) + \frac{1}{2}F(z) \tr(S) + \frac{1}{2} \langle \nabla F(z), Sz \rangle,
\end{align*}
where the equality follows from $S=\rho\E[Y\otimes Y]$.
The next lemma gives a determinant-based upper bound for the right-hand side expression.

\begin{lemma}\label{lem:logdet-domination}
Let $S\in\ptrclass$ satisfy $\norm{S}_{\op}<1$. Then, for every $z\in\cH$,
\[
  F(z)+\frac12F(z)\tr S+\frac12\langle \nabla F(z),Sz\rangle
  \le
  \det(I-S)^{-1/2}F((I-S)^{-1/2}z).
\]
\end{lemma}

\begin{proof}
Define, for $0\le t\le1$,
\[
  \Theta(t) = -\frac12\log\det(I-tS) + \rho h\roundb[\big]{\norm{(I-tS)^{-1/2}z}}.
\]
Then $e^{\Theta(0)}=F(z)$ and
$e^{\Theta(1)}=\det(I-S)^{-1/2}F((I-S)^{-1/2}z)$. Moreover,
\[
  F(z)\Theta'(0) = \frac12F(z)\tr S+\frac12\langle \nabla F(z),Sz\rangle.
\]
\cref{app:logdet-convexity} establishes that $\Theta$ is convex on $[0,1]$. Therefore
$\Theta(1)\ge\Theta(0)+\Theta'(0)$, and
\[
  e^{\Theta(1)} \ge e^{\Theta(0)+\Theta'(0)} \ge (1+\Theta'(0)) e^{\Theta(0)}.
\]
Substituting the expression for $F(z)\Theta'(0)$ proves the claim.
\end{proof}
Applying \cref{lem:logdet-domination} and multiplying by $\det(I+\rho^{-1}(G+W))^{-1/2}$ gives
\begin{align*}
    \det(I&+\rho^{-1}(G+W))^{-1/2}\E F(z+Y) \\
    &\leq \det(I+\rho^{-1}(G+W))^{-1/2}\det(I-S)^{-1/2}F((I-S)^{-1/2}z).
\end{align*}
We identify the two sides. Since $z+Y=B^{-1/2}(m+X)/\sqrt\rho$ and $B=G+W+\rho I$, the left-hand side is $\E\,U_\rho^h(m+X,G+W)$.

For the right-hand side, since $I-S=B^{-1/2}(G+\rho I)B^{-1/2}$ and $z=\rho^{-1/2}B^{-1/2}m$,
\[
  \norm{(I-S)^{-1/2}z}^2=\langle z,B^{1/2}(G+\rho I)^{-1}B^{1/2}z\rangle=\rho^{-1}\langle m,(G+\rho I)^{-1}m\rangle.
\]
Moreover, the multiplicativity of the Fredholm determinant \citep[Thm.~3.5]{simon2005trace} gives
\[
  \det(I-S)=\frac{\det(I+\rho^{-1}G)}{\det(I+\rho^{-1}(G+W))}.
\]
Together, these identities show that the right-hand side is $U_\rho^h(m,G)$. The preceding inequality therefore reads
\[
  \E U_\rho^h(m+X,G+W)\leq U_\rho^h(m,G).
\]
This proves the one-step estimate \eqref{eq:one-step}, and hence \cref{thm:bennett-logdet-potential}.

\subsection{Continuous-time marked-point-process extension (\cref{prop:continuous-bounded-jumps})}
\label{sec:proof-continuous-bounded-jumps}
We first derive a compensator drift inequality as an infinitesimal consequence of the one-step estimate, \cref{eq:one-step}. We then verify the martingale, quadratic-variation and bounded-jump assertions directly in $\cH$, establish the supermartingale property in finite dimensions by an Itô--Meyer formula,\footnotemark{} and pass to a general separable Hilbert space by finite-rank projection.

\footnotetext{A direct infinite-dimensional proof is also possible by applying the Banach-space Itô--Föllmer formula of \citet{hirai2023ito} to $(M,V)$, viewed as taking values in the product of $\cH$ and the real Banach space of self-adjoint trace-class operators, with the second coordinate equipped with the trace norm. Hirai's formula is stated for functions defined on the whole product space, whereas $U_\rho^h(m,G)$ is defined only on the open subset of pairs $(m,G)$ for which $G+\rho I$ is positive and boundedly invertible. Although $V_t\succeq0$ keeps the process inside this subset, establishing the corresponding local-domain formula and identifying its pathwise integral with the usual stochastic integral requires additional work.}

For the finite-dimensional argument, adopt the following notation. For a differentiable function $f(m,G)$ of $m\in\Rd$ and a symmetric matrix $G\in\R^{d\times d}$, write $\mathrm{D}_m f(m,G)$ and $\mathrm{D}_G f(m,G)$ for its partial Fréchet derivatives in the first and second arguments, respectively. Their evaluations in the directions $x\in\Rd$ and a symmetric matrix $A\in\R^{d\times d}$ are denoted by $\mathrm{D}_m f(m,G)[x]$ and $\mathrm{D}_G f(m,G)[A]$.

The expression in the lemma below is precisely the compensator drift produced by the Itô--Meyer formula. The lemma is the infinitesimal form of the one-step inequality~\eqref{eq:one-step}.
\begin{lemma}\label{lem:continuous-jump-drift}
Fix $\rho>0$ and $d\in\Np$. For every $m,x\in\Rd$ with $\norm{x}\leq1$ and every positive semidefinite $G\in\R^{d\times d}$,
\[
  \mathrm{D}_G U_\rho^h(m,G)[x\otimes x]+U_\rho^h(m+x,G)-U_\rho^h(m,G)-\mathrm{D}_m U_\rho^h(m,G)[x]\leq0.
\]
\end{lemma}

\begin{proof}
Fix $\epsilon\in(0,1)$, let $\xi_\epsilon$ be Bernoulli with mean $\epsilon$, and set $X_\epsilon=(\xi_\epsilon-\epsilon)x$ and $A=x\otimes x$. Then $\E X_\epsilon=0$, $\norm{X_\epsilon}\leq1$ and $\E[X_\epsilon\otimes X_\epsilon]=\epsilon(1-\epsilon)A$. Write
\[
  U^-_\epsilon=U_\rho^h(m-\epsilon x,G+\epsilon(1-\epsilon)A), \qquad U^+_\epsilon=U_\rho^h(m+(1-\epsilon)x,G+\epsilon(1-\epsilon)A).
\]
Applying \cref{eq:one-step}, the one-step potential inequality proved in \cref{sec:proof-bennett-logdet-potential}, to $X_\epsilon$ gives $(1-\epsilon)U^-_\epsilon+\epsilon U^+_\epsilon\leq U_\rho^h(m,G)$. Since $U_\rho^h(m,G)=(1-\epsilon)U_\rho^h(m,G)+\epsilon U_\rho^h(m,G)$, rearranging and dividing by $\epsilon$ yields
\[
  (1-\epsilon)\frac{U_\epsilon^--U_\rho^h(m,G)}{\epsilon}+U_\epsilon^+-U_\rho^h(m,G)\leq0.
\]
The regularisation by $\rho I$ makes $U_\rho^h$ continuously differentiable in $G$, while $h(r)=O(r^2)$ as $r\downarrow0$ gives continuous differentiability in $m$. Since $\epsilon(1-\epsilon)A=\epsilon A-\epsilon^2A$,
\[
  \frac{U^-_\epsilon-U_\rho^h(m,G)}{\epsilon}\to-\mathrm{D}_mU_\rho^h(m,G)[x]+\mathrm{D}_GU_\rho^h(m,G)[A],
\]
while continuity gives $U^+_\epsilon\to U_\rho^h(m+x,G)$. Letting $\epsilon\downarrow0$ in the preceding inequality proves the claim.
\end{proof}

\begin{proof}[Proof of \cref{prop:continuous-bounded-jumps}]We first verify that $M$ and $V$ have the stated martingale, quadratic-variation and jump properties. Since $\norm{H_t(z)\otimes H_t(z)}_{\trclass}=\norm{H_t(z)}^2$, the integrability assumption makes $V_t$ well-defined as a positive trace-class operator for every $t\geq0$. By continuity and linearity of the trace,
\[
  \tr V_T=\int_{(0,T]\times\cZ}\norm{H_t(z)}^2\,\nu(\dif t,\dif z)<\infty
\]
almost surely for every $T<\infty$.

For $x\in\cH$, write $H_t^x(z)=\langle H_t(z),x\rangle$. The scalar compensated-random-measure construction of \citet[Def.~II.1.27 and Thm.~II.1.33(a)]{jacodshiryaev2003}, applied after localisation to the integrands $H^x$ and combined with the standard Hilbert-space $L^2$ completion along an orthonormal basis, gives that $M$ is a purely discontinuous local martingale that is locally square-integrable.

For $0\leq r\leq t$, the increment $V_t-V_r$ is positive, and hence
\[
  \norm{V_t-V_r}_{\trclass}=\tr(V_t-V_r)=\int_{(r,t]\times\cZ}\norm{H_s(z)}^2\,\nu(\dif s,\dif z).
\]
It follows that $V$ is locally of finite variation in trace norm. Since $\nu$ has no time atoms, the scalar measure on the right has no time atoms, and therefore $V$ is continuous in trace norm. As $V$ is adapted and continuous, it is predictable.

Since $\nu$ has no time atoms, the predictable-covariation formula in \citet[Thm.~II.1.33(a)]{jacodshiryaev2003} gives, for $x,y\in\cH$,
\[
  \left\langle\langle M,x\rangle,\langle M,y\rangle\right\rangle_t=\int_{(0,t]\times\cZ}\langle x,H_s(z)\rangle\langle H_s(z),y\rangle\,\nu(\dif s,\dif z)=\langle x,V_ty\rangle.
\]
Thus $V$ is the operator-valued predictable quadratic variation of $M$.

The jump identity in \citet[Def.~II.1.27]{jacodshiryaev2003} and the absence of time atoms of $\nu$ give
\[
  \Delta M_t=\int_{\cZ}H_t(z)\,\mu(\{t\},\dif z).
\]
Since $\mu$ has at most one point at each time and $\norm{H_t(z)}\leq1$, we have $\norm{\Delta M_t}\leq1$.

It remains to prove the supermartingale property. Assume first that $\cH=\Rd$. Fix $\rho>0$ and define
\[
  \Gamma_\rho(m,G;x)=\mathrm{D}_GU_\rho^h(m,G)[x\otimes x]+U_\rho^h(m+x,G)-U_\rho^h(m,G)-\mathrm{D}_mU_\rho^h(m,G)[x].
\]
The finite-dimensional Itô--Meyer calculation in \cref{app:continuous-bounded-jumps} shows that $U_\rho^h(M_t,V_t)$ admits the decomposition
\begin{equation}\label{eq:continuous-ito-decomposition}
  U_\rho^h(M_t,V_t)=1+N_t+A_t,\tag{$\ddagger$}
\end{equation}
where $N$ is a local martingale and $A$ is the locally finite-variation process
\[
  A_t=\int_{(0,t]\times\cZ}\Gamma_\rho(M_{s-},V_s;H_s(z))\,\nu(\dif s,\dif z).
\]
Since $\norm{H_s(z)}\leq1$, \cref{lem:continuous-jump-drift} gives $  \Gamma_\rho(M_{s-},V_s;H_s(z))\leq0$. Hence $A$ is nonincreasing, and \eqref{eq:continuous-ito-decomposition} shows that $(U_\rho^h(M_t,V_t))_{t\geq0}$ is a local supermartingale. Since $U_\rho^h$ is nonnegative, it is a supermartingale with initial value $U_\rho^h(0,0)=1$.

It remains to pass the supermartingale property to the general separable Hilbert space $\cH$. Choose an orthonormal basis $(e_j)_{j\geq1}$ of $\cH$, let $P_k$ be the orthogonal projection onto $\spn\{e_1,\dots,e_k\}$, and define $H_t^{(k)}(z)=P_kH_t(z)$. Then $H^{(k)}$ is predictable, $\norm{H_t^{(k)}(z)}\leq\norm{H_t(z)}\leq1$, and
\[
  \int_{(0,T]\times\cZ}\norm{H_t^{(k)}(z)}^2\,\nu(\dif t,\dif z)\leq\int_{(0,T]\times\cZ}\norm{H_t(z)}^2\,\nu(\dif t,\dif z)<\infty
\]
almost surely for every $T<\infty$. The corresponding processes satisfy
\[
  M_t^{(k)}=\int_{(0,t]\times\cZ}H_s^{(k)}(z)\,(\mu-\nu)(\dif s,\dif z)=P_kM_t
\]
and
\[
  V_t^{(k)}=\int_{(0,t]\times\cZ}H_s^{(k)}(z)\otimes H_s^{(k)}(z)\,\nu(\dif s,\dif z)=P_kV_tP_k.
\]
Viewed as processes on the finite-dimensional Hilbert space $P_k\cH$, $(M^{(k)},V^{(k)})$ satisfies the assumptions already considered. Hence $(U_\rho^h(M_t^{(k)},V_t^{(k)}))_{t\geq0}$ is a nonnegative supermartingale with initial value one when $U_\rho^h$ is evaluated on $P_k\cH$.

To compare the projected potentials with $U_\rho^h(M_t,V_t)$, regard $M_t^{(k)}$ as an element of $\cH$ and $V_t^{(k)}$ as an operator on $\cH$ that vanishes on $(P_k\cH)^\perp$. Under this identification, the Hilbert-space formula for $U_\rho^h(M_t^{(k)},V_t^{(k)})$ agrees with its finite-dimensional value. Indeed, the extension of $V_t^{(k)}$ contributes only unit eigenvalues to $\det(I+\rho^{-1}V_t^{(k)})$ on $(P_k\cH)^\perp$, while $M_t^{(k)}$ has no component in that subspace.

For every fixed $t\geq0$, almost surely,
\[
  \norm{M_t^{(k)}-M_t}\to0,\qquad \norm{V_t^{(k)}-V_t}_{\trclass}\to0.
\]
The first convergence follows from the strong convergence of $P_k$. For the second, since $V_t$ is trace-class,
\[
  V_t-P_kV_tP_k=(I-P_k)V_t+P_kV_t(I-P_k),
\]
and both terms on the right converge to zero in trace norm. Continuity of the Fredholm determinant in trace norm gives
$\det(I+\rho^{-1}V_t^{(k)})\to\det(I+\rho^{-1}V_t)$. Moreover, trace-norm convergence implies operator-norm convergence, so
$(V_t^{(k)}+\rho I)^{-1/2}\to(V_t+\rho I)^{-1/2}$ in operator norm. Together with $M_t^{(k)}\to M_t$, this shows that
\[
  U_\rho^h(M_t^{(k)},V_t^{(k)})\to U_\rho^h(M_t,V_t).
\]

Fix $0\leq s\leq t$. Since the projected potentials are nonnegative, conditional Fatou's lemma and their finite-dimensional supermartingale property give
\begin{align*}
  \E[U_\rho^h(M_t,V_t)\mid\cF_s]&\leq\liminf_{k\to\infty}\E[U_\rho^h(M_t^{(k)},V_t^{(k)})\mid\cF_s] \\
  &\leq\liminf_{k\to\infty}U_\rho^h(M_s^{(k)},V_s^{(k)})=U_\rho^h(M_s,V_s).
\end{align*}
Thus $(U_\rho^h(M_t,V_t))_{t\geq0}$ is a nonnegative supermartingale with initial value $U_\rho^h(0,0)=1$.

The stated probability bound now follows from Ville's inequality. For every $T<\infty$,
\[
  \P{\exists t\in[0,T]:U_\rho^h(M_t,V_t)>e^y}\leq e^{-y}.
\]
As $T\to\infty$, the crossing events on the left increase to the corresponding event over all $t\geq0$, so continuity of probability from below preserves the same bound. Expanding the definition of $U_\rho^h$ establishes the claim.
\end{proof}

\printbibliography

\appendix

\section{Directional sub-Poisson and sub-gamma counterexample}\label{app:sub-poisson-counterexample}

For $\rho>0$, $f\in\curlyb{h,b}$, scalar $m\in\R$ and $v\geq0$, write
\[
  U_\rho^f(m,v)=(1+v/\rho)^{-1/2}\exp\curlyb[\bigg]{\rho f\roundb[\bigg]{\frac{\abs{m}}{\sqrt{\rho(v+\rho)}}}}.
\]

\begin{proposition}\label{prop:sub-poisson-counterexample}
Let $X$ have distribution
\[
  \P{X=2}=\P{X=-2}=\frac1{36},\qquad \P{X=0}=\frac{17}{18}.
\]
Then $X$ is centred with $v:=\E X^2=2/9$, and satisfies the two-sided sub-Poisson condition with variance proxy $v$:
\[
  \log\E e^{sX}=\log\E e^{-sX}\leq v\psi_h(s),\qquad s\geq0.
\]
It therefore also satisfies the two-sided sub-gamma condition with the same variance proxy:
\[
  \log\E e^{sX}=\log\E e^{-sX}\leq v\psi_b(s),\qquad 0\leq s<3.
\]
Moreover, for each $f\in\curlyb{h,b}$,
\[
  \E U_\rho^f(X,v)=1+\frac{1}{54\rho^2}+O(\rho^{-3}),\qquad \rho\to\infty,
\]
so $\E U_\rho^f(X,v)>1$ for all sufficiently large $\rho$.
\end{proposition}

\begin{proof}
The law of $X$ is symmetric, so $\E X=0$, while direct calculation gives $v=\E X^2=2/9$. For $s\geq0$,
\[
  \E e^{sX}=\frac{17+\cosh(2s)}{18}.
\]
Define
\[
  g(s)=\frac29(e^s-1-s)-\log\roundb[\bigg]{\frac{17+\cosh(2s)}{18}}.
\]
Then $g(0)=0$ and
\[
  g'(s)=\frac29(e^s-1)-\frac{2\sinh(2s)}{17+\cosh(2s)}.
\]
Writing $t=e^s\geq1$ gives
\begin{align*}
  9t^2(17+\cosh(2s))g'(s)&=(t-1)^2\roundb[\big]{t^3-8t^2+17t+8} \\
  &=(t-1)^2\roundb[\big]{t(t-4)^2+t+8}\geq0.
\end{align*}
Hence $g$ is nondecreasing on $\Rp$, and therefore $g(s)\geq0$ for every $s\geq0$. Symmetry gives the same inequality with $X$ replaced by $-X$, proving the sub-Poisson condition. Since $\psi_h(s)\leq\psi_b(s)$ for $0\leq s<3$, the sub-gamma condition follows.

It remains to analyse the potentials. The Bennett and Bernstein rate functions have the common expansion
\[
  h(u)=\frac{u^2}{2}-\frac{u^3}{6}+O(u^4) \spaced{and} b(u)=\frac{u^2}{2}-\frac{u^3}{6}+O(u^4),\qquad u\downarrow0.
\]
Write $\mu_3=\E\abs{X}^3$ and $\mu_4=\E X^4$. Substituting $\abs{X}/\sqrt{\rho(\rho+v)}$ into this expansion and using $\abs{X}\leq2$ gives, for $f\in\curlyb{h,b}$,
\[
  \rho f\roundb[\bigg]{\frac{\abs{X}}{\sqrt{\rho(\rho+v)}}}=\frac{X^2}{2\rho}-\frac{vX^2/2+\abs{X}^3/6}{\rho^2}+O(\rho^{-3}).
\]
Expanding the exponential and the factor $(1+v/\rho)^{-1/2}$ therefore yields
\[
  \E U_\rho^f(X,v)=1+\frac{\mu_4-\frac43\mu_3-3v^2}{8\rho^2}+O(\rho^{-3}),\qquad f\in\curlyb{h,b}.
\]
For the present law,
\[
  \mu_3=\frac49,\qquad \mu_4=\frac89,\qquad \mu_4-\frac43\mu_3-3v^2=\frac4{27}.
\]
Substitution proves the claimed expansion.
\end{proof}

\section{Scalar verification for \cref{lem:bennett-tangent}}
\label{app:scalar-bellman}

We complete the scalar verification deferred from the proof of \cref{lem:bennett-tangent}. Fix $r>0$, $q\geq0$ and $\rho>0$, and set
\[
  \ell=\log(1+r), \qquad D(a)=1+\frac{\rho}{2}(q^2-r^2)+\rho(\ell-r)a+\frac{\rho^2r\ell}{2}a^2.
\]
Let $\cI=[a_-,a_+]$ be the nonempty set of $a$ for which there exists $b\geq0$ satisfying
\[
  q^2=(r+a)^2+b^2, \qquad a^2+b^2\leq\rho^{-2}.
\]
By the reduction in the main text, it remains to prove
\[
  \exp\curlyb{\rho\{h(q)-h(r)\}}\leq\min_{u\in\cI}D(u).
\]
Since $D$ is a convex quadratic and $\cI$ is an interval, its minimum over $\cI$ is attained at the right endpoint, at its unconstrained minimiser $a_\star$ when $a_\star\in\cI$, or at the left endpoint. The proof will be complete once the bound is verified to hold in the three cases.

We shall use the elementary bounds
\[
  \frac{2r}{2+r}\leq\ell<r.
\]
These follow since $r-\ell$ and $\ell-2r/(2+r)$ vanish at zero and are increasing on $[0,\infty)$.

\subsection{Right endpoint verification} For every $a\in\cI$, there exists $b\geq0$ such that $q^2=(r+a)^2+b^2$, and hence $a\leq q-r$. Moreover, for any such $a$ and $b$, the reverse triangle inequality gives
\[
  \abs{q-r}=\abs{\sqrt{(r+a)^2+b^2}-r}\leq\sqrt{a^2+b^2}\leq\rho^{-1}.
\]
Thus $a=q-r$ and $b=0$ satisfy the defining constraints of $\cI$. Therefore $a_+=q-r$ is the right endpoint, and
\[
  D(a_+) = 1 + \rho \ell a_+ + \frac{\rho}{2}a_+^2 + \frac{\rho^2 r \ell}{2} a_+^2.
\]
It therefore suffices to prove that, whenever $\abs{a}\leq\rho^{-1}$ and $r+a\geq0$,
\begin{equation}\label{eq:right-endpoint-estimate}
  \exp\curlyb{\rho\{h(r+a)-h(r)\}}
  \leq
  1+\rho\ell a+\frac{\rho}{2}a^2+\frac{\rho^2r\ell}{2}a^2.
\end{equation}
We prove the inequality separately for positive and negative $a$; the case $a=0$ holds with equality.

First take $a=\delta\in(0,\rho^{-1}]$ and set $\theta=\rho\delta\in(0,1]$. Let
\[
  A_+=\frac{h(r+\delta)-h(r)}{\delta}=\frac1\delta\int_0^\delta\log(1+r+s)\dif s.
\]
Then the left-hand side of \cref{eq:right-endpoint-estimate} is $e^{\theta A_+}$, while its right-hand side is
\[
  1+\theta\ell+\frac{\rho^{-1}+r\ell}{2}\theta^2.
\]
For $A\geq0$ and $0\leq\theta\leq1$,
\[
  e^{\theta A}\leq1+\theta A+\theta^2(e^A-1-A).
\]
Indeed, this follows by comparing the power series term by term, using $\theta^k\leq\theta^2$ for every $k\geq2$. Applying this inequality with $A=A_+$, it is enough to prove
\[
  \frac{A_+-\ell}{\rho\delta}+e^{A_+}-1-A_+\leq\frac{1}{2\rho} + \frac{r\ell}{2}.
\]
Write $d=A_+-\ell$. Since $\log(1+r+s)-\log(1+r)\leq s$,
\[
  0\leq d=\frac1\delta\int_0^\delta\{\log(1+r+s)-\log(1+r)\}\,\dif s\leq\frac\delta2.
\]
Moreover, Jensen's inequality gives
\[
  e^{A_+}\leq\frac1\delta\int_0^\delta(1+r+s)\,\dif s=1+r+\frac\delta2.
\]
Using the two bounds and $A_+=\ell+d$, we therefore obtain
\[
  \frac{d}{\rho\delta} + e^{A_+} - 1 - A_+ \leq \frac{d}{\rho\delta} + r + \frac{\delta}{2} - \ell - d
  = \frac{d}{\delta}\roundb[\Big]{\frac{1}{\rho} - \delta} + r + \frac{\delta}{2} - \ell
  \leq\frac{1}{2\rho} + r - \ell
\]
where the second inequality follows from $1/\rho - \delta \geq 0$ and $d/\delta\leq1/2$. Finally, $r - \ell \leq \frac{r\ell}{2}$, by the elementary bound $\ell \geq 2r/(2+r)$. This proves the positive case.

Now take $a=-\delta$, where $0<\delta\leq\min\{r,\rho^{-1}\}$, and again set $\theta=\rho\delta$. Let
\[
  A_-=\frac{h(r)-h(r-\delta)}{\delta}=\frac1\delta\int_0^\delta\log(1+r-s)\,\dif s.
\]
Since $\delta\leq r$, we have $0\leq A_-\leq\ell$. The left-hand side of \cref{eq:right-endpoint-estimate} is $e^{-\theta A_-}$, while its right-hand side is
\[
  1-\theta\ell+\frac{\rho^{-1}+r\ell}{2}\theta^2.
\]
Using $e^{-u}\leq1-u+u^2/2$ for $u\geq0$, it is enough to show
\[
  \rho^{-1}\frac{\ell-A_-}{\delta}+\frac{A_-^2}{2}\leq\frac{\rho^{-1}+r\ell}{2}.
\]
As above,
\[
  0\leq\ell-A_-=\frac1\delta\int_0^\delta\{\log(1+r)-\log(1+r-s)\}\,\dif s\leq\frac\delta2.
\]
Also, $A_-\leq\ell<r$ gives $A_-^2\leq\ell^2\leq r\ell$. Hence the required inequality holds.

\subsection{Unconstrained minimiser verification}

Recall that the unconstrained minimiser is
\[
  a_\star=\frac{r-\ell}{\rho r\ell}.
\]
The elementary bounds on $\ell$ give $0<a_\star<\rho^{-1}$. Direct substitution gives $D(a_\star)=R(q)$, where
\[
  R(q)=1+\frac{\rho}{2}(q^2-r^2)-\frac{(r-\ell)^2}{2r\ell}.
\]
Writing $E(q):=\exp\curlyb{\rho\{h(q)-h(r)\}}$, the required inequality in this case is $E(q)\leq R(q)$ whenever $a_\star\in\cI$. We first reduce the claim to the endpoint inequalities $E(q_-)\leq R(q_-)$ and $E(q_+)\leq R(q_+)$. The lower-endpoint inequality follows from the right-endpoint estimate above. At the upper endpoint, the corresponding values of $a$ and $b$ satisfy $a^2+b^2=\rho^{-2}$, and we prove the required inequality under this equality.

Suppose that $a_\star\in\cI$. By definition, there exists $b\geq0$ such that
\[
  q^2=(r+a_\star)^2+b^2, \qquad a_\star^2+b^2\leq\rho^{-2}.
\]
The first condition and $a_\star>0$ give $q\geq r+a_\star=:q_-$. Substituting $b^2=q^2-(r+a_\star)^2$ into the second gives $q^2-r^2-2ra_\star\leq\rho^{-2}$, or equivalently
\[
  q\leq\sqrt{r^2+\rho^{-2}+2ra_\star}=:q_+.
\]
Thus
\[
  a_\star\in\cI \implies q\in[q_-,q_+],
\]
and it suffices to prove $E(q)\leq R(q)$ throughout this interval.

We show that it is enough to verify the endpoint inequalities $E(q_-)\leq R(q_-)$ and $E(q_+)\leq R(q_+)$. Differentiating gives
\[
  (R-E)'(q)=\rho\{q-E(q)\log(1+q)\}.
\]
For $q>0$,
\[
  q-E(q)\log(1+q)=\log(1+q)\left\{\frac{q}{\log(1+q)}-E(q)\right\}.
\]
Since $\log(1+q)>0$ and the logarithm is increasing, $(R-E)'(q)$ has the same sign as
\[
  \eta(q):=\log\frac{q}{\log(1+q)}-\rho\{h(q)-h(r)\}.
\]
Indeed, $\eta(q)\geq0$ if and only if $E(q)\leq q/\log(1+q)$. We shall use the following concavity property to control the sign changes of $\eta$.

\begin{lemma}\label{lem:q-log-concavity}
The function
\[
  q\mapsto\log\frac{q}{\log(1+q)}, \qquad q>0,
\]
is concave.
\end{lemma}

\begin{proof}
Write $\phi(q)=\log\{q/\log(1+q)\}$ and $L=\log(1+q)$. Direct differentiation gives
\[
  \phi''(q)=-\frac1{q^2}+\frac{L+1}{(1+q)^2L^2},
\]
so $\phi''(q)\leq0$ is equivalent to $q^2(L+1)\leq(1+q)^2L^2$. Put $u=L$, so that $q=e^u-1$. After dividing by $e^{2u}$, it remains to prove
\[
  (1-e^{-u})^2(1+u)\leq u^2.
\]
Let $\Delta(u)=u^2-(1+u)(1-e^{-u})^2$ and put $v=e^{-u}$. Then $\Delta(0)=0$ and
\[
  \Delta'(u)=2u(1-v+v^2)-(1-v^2).
\]
Using $u=-\log v\geq2(1-v)/(1+v)$,
\[
  \Delta'(u)\geq\frac{4(1-v)(1-v+v^2)}{1+v}-(1-v^2)=\frac{3(1-v)^3}{1+v}\geq0.
\]
Hence $\Delta\geq0$, proving the claim.
\end{proof}

By \cref{lem:q-log-concavity} and the convexity of $h$, the function $\eta$ is concave. Moreover, $\eta(r)=\log(r/\ell)>0$, so its nonnegative superlevel set on $[r,\infty)$ is an interval containing $r$. Hence $(R-E)'$ changes sign at most once on $[q_-,q_+]$, from nonnegative to nonpositive. The function $R-E$ is therefore first nondecreasing and then possibly nonincreasing, so its minimum over $[q_-,q_+]$ is attained at an endpoint.

At the lower endpoint, $q_-=r+a_\star$ and the corresponding value of $b$ is zero. Since $0<a_\star<\rho^{-1}$, \cref{eq:right-endpoint-estimate} with $a=a_\star$ gives \[E(q_-)\leq R(q_-).\]

At the upper endpoint, the corresponding value of $b$ satisfies
\[
  a_\star^2+b^2=q_+^2-r^2-2ra_\star=\rho^{-2}.
\]
To control this endpoint, we prove a more general estimate whenever the step-size constraint is active and the radial component is nonnegative. Let $a,b\geq0$ satisfy
\[
  q^2=(r+a)^2+b^2, \qquad a^2+b^2=\rho^{-2},
\]
and write $a=\rho^{-1}\alpha$, where $0\leq\alpha\leq1$. Then
\[
  q^2=r^2+\rho^{-2}+2r\rho^{-1}\alpha, \qquad D(a)=1+\ell\alpha+\frac1{2\rho}+\frac{r\ell}{2}\alpha^2.
\]
We claim that
\begin{equation}\label{eq:step-boundary-estimate}
  \exp\curlyb{\rho\{h(q)-h(r)\}}\leq1+\ell\alpha+\frac1{2\rho}+\frac{r\ell}{2}\alpha^2.
\end{equation}

If $\alpha=1$, then $b=0$ and $q=r+\rho^{-1}$, so \cref{eq:step-boundary-estimate} follows from \cref{eq:right-endpoint-estimate} with $a=\rho^{-1}$. Suppose that $0\leq\alpha<1$, and set $\beta=\rho(q-r)$. The expression for $q^2$ above gives $r<q<r+\rho^{-1}$, so $0<\beta<1$. Comparing it with $q=(r+\rho^{-1}\beta)$ gives
\[
  \rho^{-1}(1-\beta^2)=2r(\beta-\alpha),
\]
and hence $\alpha<\beta$. Let
\[
  \bar r:=\frac{r+q}{2}=r+\frac{\beta}{2\rho}=r\frac{1-\alpha\beta}{1-\beta^2}
\]
and define the secant slope
\[
  A:=\frac{h(q)-h(r)}{q-r}=\frac1{q-r}\int_r^q\log(1+u)\,\dif u.
\]
Since $\rho(q-r)=\beta$, the left-hand side of \cref{eq:step-boundary-estimate} is $e^{\beta A}$. Jensen's inequality gives
\[
  e^A\leq\frac1{q-r}\int_r^q(1+u)\,\dif u=1+\bar r,
\]
and hence $e^{\beta A}\leq(1+\bar r)^\beta$. It remains to compare this expression with the quadratic on the right-hand side of \cref{eq:step-boundary-estimate}. The following lemma gives the required bound; its proof is deferred until the conclusion of the present argument.

\begin{lemma}\label{lem:chord-comparison}
Let $0\leq\alpha\leq\beta<1$ and $r\geq0$, write $\ell=\log(1+r)$, and set
\[
  \bar r=r\frac{1-\alpha\beta}{1-\beta^2}.
\]
Then
\[
  (1+\bar r)^\beta\leq1+\beta\bar r-(r-\ell)\alpha+\frac{r\ell}{2}\alpha^2.
\]
\end{lemma}

The relation between $\alpha$ and $\beta$ also gives
\[
  \beta\bar r-r\alpha=r(\beta-\alpha)+\frac{\beta^2}{2\rho}=\frac1{2\rho}.
\]
Applying \cref{lem:chord-comparison},
\[
  e^{\beta A}\leq(1+\bar r)^\beta\leq1+\beta\bar r-(r-\ell)\alpha+\frac{r\ell}{2}\alpha^2=1+\ell\alpha+\frac1{2\rho}+\frac{r\ell}{2}\alpha^2.
\]
This proves \cref{eq:step-boundary-estimate}. Taking $a=a_\star$ and $\alpha=\rho a_\star$ therein gives \[E(q_+)\leq R(q_+).\] The endpoint reduction now yields
\[
  a_\star\in\cI\implies\exp\curlyb{\rho\{h(q)-h(r)\}}\leq D(a_\star),
\]
as desired.

\begin{proof}[Proof of \cref{lem:chord-comparison}]
Set $c=(1-\alpha\beta)/(1-\beta^2)$, so that $\bar r=cr$, and define
\[
  \Phi(r)=1+\beta cr-(1+cr)^\beta-(r-\ell)\alpha+\frac{r\ell}{2}\alpha^2.
\]
It suffices to show that $\Phi(r)\geq0$. We have $\Phi(0)=\Phi'(0)=0$, while
\[
  \Phi''(r)=\beta(1-\beta)c^2(1+cr)^{\beta-2}-\frac{\alpha}{(1+r)^2}+\frac{\alpha^2}{2}\left(\frac1{1+r}+\frac1{(1+r)^2}\right).
\]
Put $u=1+r$ and define
\[
  H(u):=\beta(1-\beta)c^2\frac{u^2}{\{1+c(u-1)\}^{2-\beta}}.
\]
Since
\[
  -\frac{\alpha}{u^2}+\frac{\alpha^2}{2}\left(\frac1u+\frac1{u^2}\right)\geq-\frac{\alpha(1-\alpha)}{u^2},
\]
it is enough to prove that $H(u)\geq\alpha(1-\alpha)$ for $u\geq1$.

The case $\alpha=0$ is immediate. Suppose that $0<\alpha\leq\beta<1$. The logarithmic derivative of $H$ has the sign of $\beta cu-2(c-1)$. It changes sign at most once, from negative to positive, so the minimum of $H$ over $u\geq1$ is attained either at $u=1$ or at the unique interior stationary point.

At $u=1$,
\[
  H(1)-\alpha(1-\alpha)=\frac{(\beta-\alpha)\{1-\alpha(1+\beta-\beta^2)\}}{(1-\beta)(1+\beta)^2}\geq0.
\]
Indeed, $\alpha\leq\beta$ and
\[
  1-\alpha(1+\beta-\beta^2)\geq1-\beta(1+\beta-\beta^2)=(1-\beta)^2(1+\beta)\geq0.
\]
This proves the required bound when the minimum of $H$ is attained at $u=1$.

It remains to consider the case where the stationary point of $H$ lies in $(1,\infty)$. This point is
\[
  u_\star=\frac{2(c-1)}{\beta c}=\frac{2(\beta-\alpha)}{1-\alpha\beta}.
\]
The condition $u_\star>1$ is equivalent to $2\beta-1>\alpha(2-\beta)$. Since $\alpha>0$, this can occur only when $\beta>1/2$; in that case, it is equivalent to
\[
  0<\alpha<\frac{2\beta-1}{2-\beta}=:\alpha_c.
\]
We now verify the required bound throughout this remaining range of parameters.

At the stationary point,
\[
  H(u_\star)=M_\beta(\alpha):=\frac{4\beta(1-\beta)^{1-\beta}(\beta-\alpha)^\beta}{(1+\beta)^\beta(2-\beta)^{2-\beta}}.
\]
Thus it remains to prove $M_\beta(\alpha)\geq\alpha(1-\alpha)$ for $0<\alpha<\alpha_c$. Define
\[
  Q(\alpha)=\frac{M_\beta(\alpha)}{\alpha(1-\alpha)}.
\]
Then
\[
  \frac{\dif}{\dif\alpha}\log Q(\alpha)=\frac{N_\beta(\alpha)}{\alpha(\beta-\alpha)(1-\alpha)}, \qquad N_\beta(\alpha):=(\beta-2)\alpha^2+(\beta+1)\alpha-\beta.
\]
The denominator is positive on $(0,\alpha_c]$. Moreover, $N_\beta'$ is decreasing and
\[
  N_\beta'(\alpha_c)=3(1-\beta)>0,
\]
so $N_\beta$ is increasing on $[0,\alpha_c]$. Since $N_\beta(\alpha_c)=\beta-1<0$, it follows that $N_\beta(\alpha)<0$ throughout this interval. Hence $Q$ is decreasing on $(0,\alpha_c]$.

Finally,
\[
  M_\beta(\alpha_c)=\frac{4\beta(1-\beta)}{(2-\beta)^2}, \qquad \alpha_c(1-\alpha_c)=\frac{3(2\beta-1)(1-\beta)}{(2-\beta)^2},
\]
and therefore
\[
  Q(\alpha_c)=\frac{4\beta}{3(2\beta-1)}\geq1.
\]
Since $Q$ is decreasing, $Q(\alpha)\geq Q(\alpha_c)\geq1$ for every $0<\alpha<\alpha_c$. Thus $H(u_\star)\geq\alpha(1-\alpha)$.

We have now controlled the minimum of $H$ whether it is attained at $u=1$ or at the interior stationary point. Hence $\Phi''(r)\geq0$ for every $r\geq0$. Since $\Phi(0)=\Phi'(0)=0$, it follows that $\Phi(r)\geq0$.
\end{proof}

\subsection{Left endpoint verification}

The right endpoint and the case $a_\star\in\cI$ have already been treated. It remains to consider $a_\star<a_-$, in which case the minimum of $D$ over $\cI$ is attained at $a_-$. Since $a_\star>0$, we have $a_->0$.

For $a\in\cI$, the identity $b^2=q^2-(r+a)^2\geq0$ gives $a\geq-q-r$, while
\[
  a^2+b^2=q^2-r^2-2ra\leq\rho^{-2}
\]
gives
\[
  a\geq\frac{q^2-r^2-\rho^{-2}}{2r}.
\]
These are the two lower constraints defining $\cI$. Since $-q-r\leq0<a_-$, the first cannot determine the left endpoint. Hence
\[
  a_-=\frac{q^2-r^2-\rho^{-2}}{2r}.
\]
For the corresponding value of $b$, this is equivalent to $a_-^2+b^2=\rho^{-2}$.

Set $\alpha=\rho a_-$. Then $0<\alpha\leq1$ and $q^2=r^2+\rho^{-2}+2r\rho^{-1}\alpha$. The estimate \cref{eq:step-boundary-estimate} therefore gives
\[
  \exp\curlyb{\rho\{h(q)-h(r)\}} \leq 1+\ell\alpha+\frac1{2\rho}+\frac{r\ell}{2}\alpha^2 = D(a_-).
\]
This proves the left-endpoint case. Together with the right-endpoint and unconstrained-minimiser cases, it completes the scalar verification required in the proof of \cref{lem:bennett-tangent}.

\section{Convexity used in the proof of \cref{lem:logdet-domination}}
\label{app:logdet-convexity}

\begin{lemma}\label{lem:theta-convex}
Fix $\rho > 0$. Let $S\in\ptrclass$ satisfy $\norm{S}_{\op}<1$, and let $z\in\cH$. Then the function
\[
  t\mapsto
  -\frac12\log\det(I-tS)
  +
  \rho h\roundb[\big]{\norm{(I-tS)^{-1/2}z}}
\]
is convex on $[0,1]$.
\end{lemma}

\begin{proof}
Since $S\in\ptrclass$, it is compact and self-adjoint. Let $(\lambda_j)_{j\geq1}$ be its nonzero eigenvalues, counted with multiplicity. For every $t\in[0,1]$ and $j\geq1$,
\[
  0\leq t\lambda_j\leq\norm{S}_{\op}<1,
\]
so
\[
  -\frac12\log\det(I-tS)=\frac12\sum_{j\geq1}-\log(1-t\lambda_j)<\infty.
\]
Each finite partial sum is convex in $t$, and the series is their pointwise supremum. Hence $t\mapsto-\frac12\log\det(I-tS)$ is convex.

It remains to consider the second term. Set
\[
  A_t=(I-tS)^{-1}, \qquad R(t)=\langle z,A_tz\rangle=\norm{(I-tS)^{-1/2}z}^2.
\]
Since $\norm{S}_{\op}<1$, the map $t\mapsto A_t$ is twice differentiable in operator norm on $[0,1]$. Moreover, $A_t$ commutes with $S$, and differentiating $(I-tS)A_t=I$ gives
\[
  A_t'=SA_t^2, \qquad A_t''=2S^2A_t^3.
\]
It follows that
\[
  R'(t)=\langle z,SA_t^2z\rangle, \qquad R''(t)=2\langle z,S^2A_t^3z\rangle.
\]
Since $S$ and $A_t$ are commuting positive operators, $R'(t),R''(t)\geq0$. Furthermore,
\[
  R'(t)=\langle A_t^{1/2}z,SA_t^{3/2}z\rangle,
\]
so the Cauchy--Schwarz inequality gives
\[
  (R'(t))^2
  \leq \norm{A_t^{1/2}z}^2\norm{SA_t^{3/2}z}^2
  =\frac12R(t)R''(t).
\]

If $z=0$, the second term is constant. Otherwise, define
\[
  r(t)=\sqrt{R(t)}=\norm{(I-tS)^{-1/2}z}.
\]
Then
\[
  r'(t)=\frac{R'(t)}{2R(t)^{1/2}}\geq0
\]
and
\[
  r''(t)=\frac{2R(t)R''(t)-(R'(t))^2}{4R(t)^{3/2}}\geq0.
\]
Thus $r$ is nondecreasing and convex. Since
\[
  h'(u)=\log(1+u)\geq0, \qquad h''(u)=\frac1{1+u}>0,
\]
we have
\[
  \frac{\dif^2}{\dif t^2}h(r(t))
  =h''(r(t))(r'(t))^2+h'(r(t))r''(t)
  \geq0.
\]
Hence $t\mapsto\rho h(\norm{(I-tS)^{-1/2}z})$ is convex. Adding the two convex terms proves the claim.
\end{proof}

\section{Itô--Meyer calculation for \cref{prop:continuous-bounded-jumps}}
\label{app:continuous-bounded-jumps}

We verify \eqref{eq:continuous-ito-decomposition}. Suppose that $\cH=\Rd$ and fix $\rho>0$. Identify the space of symmetric $d\times d$ matrices with $\R^{d(d+1)/2}$. Since $M$ is an $\Rd$-valued local martingale and $V$ is a continuous finite-variation process, the joint process $(M,V)$ is a finite-dimensional càdlàg semimartingale.

The multidimensional Itô--Meyer formula of \citet[Thm.~27.1]{metivier1982semimartingales} requires the function to which it is applied to be $C^2$. The maps $G\mapsto\det(I+\rho^{-1}G)$ and $G\mapsto(G+\rho I)^{-1/2}$ are smooth on the open set of symmetric matrices $G$ for which $G+\rho I$ is positive definite. Moreover, the expansions $h'(r)=r+O(r^2)$ and $h''(r)=1+O(r)$ as $r\downarrow0$ show that the radial dependence on $m$ is $C^2$ at $m=0$. Hence $(m,G)\mapsto U_\rho^h(m,G)$ is $C^2$ on
\[
  \cO_\rho=\curlyb{(m,G):G+\rho I\text{ is positive definite}},
\]
which contains the range of $(M,V)$ because $V_t\succeq0$.

We next localise the process so that the Itô--Meyer formula may be applied to a globally $C^2$ function. For $R\in\Np$, define
\[
  \tau_R=\inf\curlyb[\big]{t\geq0:\norm{M_t}\geq R\ \text{or}\ \tr V_t\geq R}
\]
and
\[
  \cK_R=\curlyb[\big]{(m,G):\norm{m}\leq R+1,\ G\succeq0,\ \tr G\leq R}.
\]
For $s\leq\tau_R$, we have $\norm{M_{s-}}\leq R$, $\norm{H_s(z)}\leq1$ and $\tr V_s\leq R$. Hence $(M_{s-},V_s)$ and $(M_{s-}+H_s(z),V_s)$ lie in $\cK_R$, as does $(M_{t\wedge\tau_R},V_{t\wedge\tau_R})$. The set $\cK_R$ is compact and contained in $\cO_\rho$, since $G+\rho I\succeq\rho I$ whenever $G\succeq0$.

Choose a smooth cutoff supported in $\cO_\rho$ and equal to one on a neighbourhood of $\cK_R$. Multiplying $U_\rho^h$ by this cutoff on $\cO_\rho$ and extending the product by zero outside $\cO_\rho$ gives a globally $C^2$ extension of $U_\rho^h$ from a neighbourhood of $\cK_R$. This extension agrees with $U_\rho^h$, together with its first two derivatives, on $\cK_R$.

Apply the Itô--Meyer formula to this globally $C^2$ extension of $U_\rho^h$ and the stopped semimartingale $(M_{t\wedge\tau_R},V_{t\wedge\tau_R})_{t\geq0}$. Since the extension agrees with $U_\rho^h$ at every point appearing in the formula, we write the resulting identity in terms of $U_\rho^h$. Since $M$ is purely discontinuous and $V$ is of finite variation, the joint process has no continuous local-martingale part, so the continuous quadratic-covariation term vanishes. Since $V$ is continuous, $\Delta V_s=0$. The formula therefore gives
\begin{align*}
  U_\rho^h(M_{t\wedge\tau_R},V_{t\wedge\tau_R})
  &=
  1+\int_{(0,t\wedge\tau_R]}\mathrm{D}_mU_\rho^h(M_{s-},V_s)[\dif M_s]
  +\int_{(0,t\wedge\tau_R]}\mathrm{D}_GU_\rho^h(M_{s-},V_s)[\dif V_s]\\
  &+\sum_{0<s\leq t\wedge\tau_R}\curlyb[\big]{U_\rho^h(M_s,V_s)-U_\rho^h(M_{s-},V_s)-\mathrm{D}_mU_\rho^h(M_{s-},V_s)[\Delta M_s]}.
\end{align*}

At an atom $(s,z)$ of $\mu$, the absence of time atoms of $\nu$ and the assumption that $\mu$ has at most one point at each time give $\Delta M_s=H_s(z)$. Since $V$ is continuous, $V_s=V_{s-}$. Therefore the sum over jumps in the Itô--Meyer formula is exactly
\[
  \int_{(0,t\wedge\tau_R]\times\cZ}\curlyb[\big]{U_\rho^h(M_{s-}+H_s(z),V_s)-U_\rho^h(M_{s-},V_s)-\mathrm{D}_mU_\rho^h(M_{s-},V_s)[H_s(z)]}\,\mu(\dif s,\dif z).
\]

Decompose this integral as the sum of its integrals against $\mu-\nu$ and $\nu$. Combining the former with the integral against $M$ cancels the two terms involving $\mathrm{D}_mU_\rho^h$, while combining the latter with the integral against $V$ gives the integrand $\Gamma_\rho$. Hence
\begin{align*}
  &U_\rho^h(M_{t\wedge\tau_R},V_{t\wedge\tau_R}) \\
  &\qquad=1+\int_{(0,t\wedge\tau_R]\times\cZ}\curlyb[\big]{U_\rho^h(M_{s-}+H_s(z),V_s)-U_\rho^h(M_{s-},V_s)}\,(\mu-\nu)(\dif s,\dif z)\\
  &\qquad\qquad+\int_{(0,t\wedge\tau_R]\times\cZ}\Gamma_\rho(M_{s-},V_s;H_s(z))\,\nu(\dif s,\dif z)
  =:1+N_t^{\tau_R}+A_t^{\tau_R}.
\end{align*}

We now verify that $N^{\tau_R}$ is a square-integrable martingale and that $A^{\tau_R}$ is well-defined and has finite variation. On the stopped range, boundedness of the first and second derivatives of the cutoff extension, together with the mean-value theorem and the second-order Taylor remainder, gives
\[
  \abs{U_\rho^h(m+x,G)-U_\rho^h(m,G)}\leq C_R\norm{x},\qquad
  \abs{\Gamma_\rho(m,G;x)}\leq C_R\norm{x}^2
\]
whenever $\norm{m}\leq R$, $G\succeq0$, $\tr G\leq R$ and $\norm{x}\leq1$. Consequently,
\[
  \int_{(0,t\wedge\tau_R]\times\cZ}\abs{U_\rho^h(M_{s-}+H_s(z),V_s)-U_\rho^h(M_{s-},V_s)}^2\,\nu(\dif s,\dif z)
  \leq C_R^2\tr V_{t\wedge\tau_R}\leq C_R^2R,
\]
so $N^{\tau_R}$ is a square-integrable martingale. Similarly,
\[
  \int_{(0,t\wedge\tau_R]\times\cZ}\abs{\Gamma_\rho(M_{s-},V_s;H_s(z))}\,\nu(\dif s,\dif z)
  \leq C_R\tr V_{t\wedge\tau_R}\leq C_RR,
\]
so $A^{\tau_R}$ is well-defined and has finite variation.

The first estimate shows that the compensated integral in the preceding decomposition defines a local martingale $N$, with localising sequence $(\tau_R)$, while the second shows that the process $A$ defined in the main proof is well-defined and locally of finite variation. Hence
\[
  U_\rho^h(M_{t\wedge\tau_R},V_{t\wedge\tau_R})
  =
  1+N_{t\wedge\tau_R}+A_{t\wedge\tau_R}.
\]
Since $M$ has càdlàg paths and $\tr V_T<\infty$ for every $T<\infty$, we have $\tau_R\uparrow\infty$ almost surely. Thus, on every finite time interval, the stopped identity agrees with the unstopped identity for all sufficiently large $R$, proving \eqref{eq:continuous-ito-decomposition}.

\end{document}